\documentclass[a4paper]{article}

\usepackage{amsfonts}
\usepackage{amsmath}
\usepackage{mathrsfs}
\usepackage{graphicx}
\usepackage{amsmath,amsthm,amssymb,amscd,,amstext,amsfonts}
\usepackage{color}
\usepackage{enumerate}
\usepackage{tikz}

\usepackage[colorlinks=true, linkcolor=blue, urlcolor=red, citecolor=blue, hyperindex]
{hyperref}

\renewcommand{\baselinestretch}{1.05}

\theoremstyle{plain}\newtheorem{maintheorem}{Theorem}

\newtheorem{maincorollary}{Corollary}

\newtheorem{mainproposition}{Proposition}

\newtheorem{mainlemma}{Lemma}

\newtheorem{Thm}{Theorem}[section]
\newtheorem{Lem}[Thm]{Lemma}
\newtheorem{Prop}[Thm]{Proposition}

\theoremstyle{remark}
\newtheorem{Def}[Thm] {Definition}
\newtheorem{Rem}[Thm] {Remark}

\newtheorem{Que}[Thm] {Question}

\DeclareMathOperator{\card}{card}
\DeclareMathOperator{\orb}{orb}

\DeclareMathOperator{\dist}{dist}
\DeclareMathOperator{\cov}{cov}
\DeclareMathOperator{\interior}{Int}
\DeclareMathOperator{\Span}{span}
\DeclareMathOperator{\Int}{Int}

\newcommand{\eps}{\varepsilon}

\newcommand{\N}{\mathbb{N}}
\newcommand{\Z}{\mathbb{Z}}
\newcommand{\set}[1]{\left\{#1\right\}}

\newcommand{\htop}{h_{\text{top}}}

\newcommand{\M}{\mathfrak{M}}
\newcommand{\E}{\mathcal{E}}

\newcommand{\B}{\mathcal{B}}
\newcommand{\C}{\mathfrak{C}}

\begin{document}

\title{Different Statistical Future of Dynamical Orbits over Expanding or Hyperbolic
Systems (II): Nonempty Syndetic Center}


\author{{Yiwei Dong$^{1}$ and Xueting Tian$^{2}$ $^{3}$}\\
{\em\small $^{1,2}$ School of Mathematical Science,  Fudan University}\\
{\em\small Shanghai 200433, People's Republic of China}\\
{\small Email: dongyiwei06@gmail.com;     xuetingtian@fudan.edu.cn}\\
}
\date{}
\maketitle

\renewcommand{\baselinestretch}{1.2}
\large\normalsize
\footnotetext {$^{3}$ Tian is the corresponding author.}
\footnotetext {AMS Review:   37D20;  37C50;  37B20;  37B40;  37C45.   }

\begin{abstract}

In \cite{DongTian2016-nosyndetic} different statistical behavior of dynamical orbits without syndetic center are considered. In  present paper we continue this project and consider different statistical behavior of dynamical orbits with nonempty syndetic center: Two of sixteen cases appear (for which other fourteen cases are still unknown) in transive topologically expanding or hyperbolic systems and are discovered to have full topological entropy for which it is also true if combined with non-recurrence and  multifractal analysis such as quasi-regular set, irregular set and level sets. In this process a strong entropy-dense property, called minimal-entropy-dense,  is established.

 In particular, we show that points that are minimal  (or called almost periodic), a classical and important concept in the study of dynamical systems, form a set with full topological entropy if the dynamical system satisfies shadowing or almost specification property.

\end{abstract}

\tableofcontents


\section{Introduction}

In this paper, a topological dynamic system $(X,f)$ means a continuous map $f$ operate on a compact metric space $(X,d)$.
  In \cite{DongTian2016-nosyndetic}, the authors introduced several concepts on statistical $\omega$-limit sets by using natural density and Banach density and used them to differ asymptotic behavior of dynamical orbits. In the case of dynamical orbits without syndetic center, it is proved in \cite{DongTian2016-nosyndetic} that exactly twelve different asymptotic behavior  appear  and they all carry full topological entropy when the system is transitive and has expansiveness and shadowing. In present paper we continue this program mainly to consider dynamical orbits with syndetic center.

\subsection{Statistical $\omega-$limit sets}



Throughout this paper,  we denote the sets of natural numbers,   integer numbers and nonnegative numbers by $\N,  \Z,  \Z^+$ respectively.
Now we recall some notions of density and then we use them to describe different statistical structure of dynamical orbits.
Let $S\subseteq \mathbb{N}$,
 define
$$\bar{d} (S):=\limsup_{n\rightarrow\infty}\frac{|S\cap \{0,  1,  \cdots,  n-1\}|}n,  \,  \,
  \underline{d} (S):=\liminf_{n\rightarrow\infty}\frac{|S\cap \{0,  1,  \cdots,  n-1\}|}n,  $$
where $|Y|$ denotes the cardinality of the set $Y$.
 These two concepts are called {\it upper density} and {\it lower density} of $S$,   respectively.
 If $\bar{d} (S)=\underline{d}(S)=d,  $ we call $S$ to have density of $d.  $
 Define
  $$B^* (S):=\limsup_{|I|\rightarrow\infty}\frac{|S\cap I|}{|I|},  \,  \,   B_* (S):=\liminf_{|I|\rightarrow\infty}\frac{|S\cap I|}{|I|},  $$
here $I\subseteq \mathbb{N}$ is taken from finite continuous integer intervals.
These two concepts are called
{\it Banach upper density}
and
{\it Banach lower density} of $S$,   respectively.
 A set $S\subseteq \mathbb{N}$ is called {\it syndetic},
  if there is $N\in\N$ such that for any $n\in\N,  $ $ S\cap \{n,  n+1,  \cdots,  n+N\}\neq \emptyset.    $
 These concepts of density are basic and have played important roles in the field of dynamical systems,   ergodic theory and number theory,   etc.
Let $U,  V\subseteq X$ be two nonempty open subsets and $x\in X.  $
 Define
  sets of visiting time
 $$N (U,  V):=\{n\geq 1|\,  U\cap f^{-n} (V)\neq \emptyset\} \,  \,  \text{ and } \,  \,  N (x,  U):=\{n\ge 1|\,  f^n (x)\in U\}.  $$
 By using these concepts of density and visiting time,
 Dong and Tian  defined some statistical $\omega$-limit sets in \cite{DongTian2016-nosyndetic} as follows.

\begin{Def} (Statistical $\omega-$limit sets)
 For $x\in X$ and $\xi=\overline{d},   \,  \underline{d},   \,  B^*,  \,   B_*$,   a point $y\in X$ is called $x-\xi-$accessible,   if for any $ \epsilon>0,  \,  N (x,  V_\epsilon (y))\text{ has positive   density w.  r.  t.   }\xi,  $ where     $V_\epsilon (x)$ denotes  the  ball centered at $x$ with radius $\epsilon$.
 Let $$\omega_{\xi}(x):=\{y\in X\,  |\,   y\text{ is } x-\xi-\text{accessible}\}.  $$  For convenience,   it is called {\it $\xi-\omega$-limit set of $x$}.
 \end{Def}





For any $x\in X$,   the orbit of $x$ is $\{f^nx\}_{n=0}^{\infty}$ which we denote by $orb(x,  f)$.
 Let $$\omega_f(x):=\bigcap_{n\geq1}\overline{\bigcup_{k\geq n}\{f^kx\}}=\{y\in X: \exists~n_i\to\infty~s.  t.  ~f^{n_i}x\to y\}.  $$
It is clear that $\omega_f(x)$ is a nonempty compact $f$-invariant set.
 Note that
\begin{equation}\label{X-three-omega}
  \omega_{B_*}(x)\subseteq \omega_{\underline{d}}(x)\subseteq \omega_{\overline{d}}(x)\subseteq \omega_{B^*}(x)\subseteq \omega_f(x).
\end{equation}
Moreover,   they are all compact and invariant (with possibility that some sets are empty). If $\omega_{B_*}(x)=\emptyset,$ there are twelve different statistical behavior and they have been studied in the previous article \cite{DongTian2016-nosyndetic}.
If $\omega_{B_*}(x)\neq\emptyset,$ it is clear that $x$ belongs to one of following sixteen cases with different statistical behavior.

\begin{Thm}\label{thm-14-cases} Suppose $(X,  f)$ is a topological dynamical system. If $\omega_{B_*}(x)\neq\emptyset,$ then
 each point $x\in X$ satisfies only one of following sixteen cases:
  \begin{description}

   \item[Case  (1).    ]  \,  \,   $ \emptyset\neq\omega_{B_*}(x)= \omega_{\underline{d}}(x)= \omega_{\overline{d}}(x)
= \omega_{B^*}(x)= \omega_f(x) ;$
 \item[Case  (2).    ]  \,  \,   $ \emptyset\neq\omega_{B_*}(x)= \omega_{\underline{d}}(x)= \omega_{\overline{d}}(x)
= \omega_{B^*}(x)\subsetneq \omega_f(x);  $

  \item[Case  (3).    ] \,   $  \emptyset \neq \omega_{B_*}(x)\subsetneq\omega_{\underline{d}}(x)= \omega_{\overline{d}}(x)= \omega_{B^*}(x)= \omega_f(x);$
  \item[Case  (4).    ]  \,  \,   $  \emptyset \neq \omega_{B_*}(x)\subsetneq\omega_{\underline{d}}(x)= \omega_{\overline{d}}(x)= \omega_{B^*}(x)\subsetneq\omega_f(x);$

  \item[Case  (5).    ]  \,  \,   $\emptyset \neq \omega_{B_*}(x)\subsetneq    \omega_{\underline{d}}(x)=
\omega_{\overline{d}}(x)  \subsetneq \omega_{B^*}(x)= \omega_f(x) ;$
\item[Case  (6).    ]  \,  \,   $\emptyset \neq \omega_{B_*}(x)\subsetneq  \omega_{\underline{d}}(x)=
\omega_{\overline{d}}(x)  \subsetneq \omega_{B^*}(x)\subsetneq \omega_f(x) ;$

  \item[Case  (7).    ]  \,  \,   $    \emptyset \neq \omega_{B_*}(x)=\omega_{\underline{d}}(x)\subsetneq
\omega_{\overline{d}}(x)= \omega_{B^*}(x)= \omega_f(x) ;$
  \item[Case  (8).    ]  \,  \,   $\emptyset \neq \omega_{B_*}(x)=\omega_{\underline{d}}(x)\subsetneq
\omega_{\overline{d}}(x)= \omega_{B^*}(x)\subsetneq \omega_f(x) ;$
\item[Case  (9).    ]  \,  \,
$\emptyset \neq \omega_{B_*}(x)\subsetneq  \omega_{\underline{d}}(x)\subsetneq
\omega_{\overline{d}}(x)= \omega_{B^*}(x)= \omega_f(x) ;$
 \item[Case  (10).    ]  \,  \,
$\emptyset \neq \omega_{B_*}(x)\subsetneq  \omega_{\underline{d}}(x)\subsetneq
\omega_{\overline{d}}(x)= \omega_{B^*}(x)\subsetneq \omega_f(x) ;$

  \item[Case  (11).    ]  \,  \,   $\emptyset \neq \omega_{B_*}(x)=  \omega_{\underline{d}}(x) \subsetneq \omega_{\overline{d}}(x)\subsetneq \omega_{B^*}(x)= \omega_f(x);$
\item[Case  (12).    ]  \,  \,   $\emptyset \neq \omega_{B_*}(x)=  \omega_{\underline{d}}(x) \subsetneq \omega_{\overline{d}}(x)\subsetneq \omega_{B^*}(x)\subsetneq \omega_f(x);$
     \item[Case  (13).    ]  \,  \,
$\emptyset \neq \omega_{B_*}(x)\subsetneq \omega_{\underline{d}}(x) \subsetneq \omega_{\overline{d}}(x)\subsetneq \omega_{B^*}(x)= \omega_f(x);$

 \item[Case  (14).    ]  \,  \,
$\emptyset \neq \omega_{B_*}(x)\subsetneq \omega_{\underline{d}}(x) \subsetneq \omega_{\overline{d}}(x)\subsetneq \omega_{B^*}(x)\subsetneq \omega_f(x);$


  \item[Case  (15).    ]  \,  \,   $\emptyset \neq \omega_{B_*}(x)=  \omega_{\underline{d}}(x) = \omega_{\overline{d}}(x)\subsetneq \omega_{B^*}(x)= \omega_f(x);$
\item[Case  (16).    ]  \,  \,   $\emptyset \neq \omega_{B_*}(x)=  \omega_{\underline{d}}(x) = \omega_{\overline{d}}(x)\subsetneq \omega_{B^*}(x)\subsetneq \omega_f(x).$
  \end{description}

 \end{Thm}

 In this paper we mainly discuss the complexity on sets of points with asymptotical  behavior of Case (1) and Case (2) (For others, it is still unknown whether they appear or not in `good' dynamical systems).

\begin{Def}  A topological dynamical system $(X,  f)$ is called {\it topologically expanding},   if   $X$ has infinitely many points, $f$  is positively expansive and satisfies the shadowing property.   A homeomorphism $(X,  f)$ is called {\it topologically hyperbolic},   if $X$ has infinitely many points,  $f$ is expansive and satisfies the shadowing property.
\end{Def}

\begin{Rem} If $X$ has finite many points, then the topological dynamical system $(X,  f)$ is simple to study so that in present paper we  always require that $X$ has infinitely many points.

\end{Rem}


 We will combine (non-)recurrence with above statistical $\omega$-limit sets to state the first main result.  We call $x\in X$ to be {\it recurrent},   if  $x\in \omega_f (x).  $ Otherwise,   $x$ is called {non-recurrent.}
 A point $x\in X$ is called {\it wandering},   if there is a neighborhood $U$ of $x$ such that the sets $f^{-n}U,  \,  \,  n\geq 0$,   are mutually disjoint.   Otherwise,   $x$ is called {non-wandering.}
Let $Rec(f)$, $NRec(f)$ and  $\Omega(f)$ denote the sets of recurrent  points, nonrecurrent points,  non-wandering points
respectively.
Fix an arbitrary $x\in X.  $  Let  $V_\epsilon (x)$ denote a ball centered at $x$ with radius $\epsilon$.   Then it is easy to check that
\begin{eqnarray*}
x   \in Rec(f) &\Leftrightarrow& \,  \forall \,  \epsilon>0,  \,  N (x,  V_\epsilon (x))\neq \emptyset,  \\
x \in \Omega(f) &\Leftrightarrow& \,   \forall \,  \epsilon>0,  \,  N (V_\epsilon (x),  V_\epsilon (x))\neq \emptyset.
\end{eqnarray*}
Other kind of recurrence such as weak almost periodic,   quasi-weak almost periodic and Banach recurrent have been discussed in \cite{HTW,  T16}. One can find more information.

 Let $\Lambda\subseteq X$ and denote by $M_{erg}(f,\Lambda)$ the set of all ergodic measures supported on $\Lambda.$  Define $\mathbb{E}_k:=\{x\in X\,  |\,  \# M_{erg}(f, \omega_f(x))=k\}$ where $\# A$ denotes the cardinality of   $A$.  A point  $x\in X$ is generic for some invariant measure  $\mu$ means that   Birkhoff averages of any
   continuous function  converge to the integral of $\mu$ (roughly speaking, time average $=$ space average).    Let $G_\mu$   denote the set of all generic points for  $\mu $.
    Let $QR(f)  = \cup_{\mu\in M(f,X)} G_\mu,$ where $M(f,X)$ denotes the space of all invariant measures of $f.$  The points in $QR(f)$ are called quasiregular points of   $f$ in \cite{DGS,Oxtoby}.

\begin{maintheorem}\label{thm-A0000000000}
  Suppose $(X,  f)$ is topologically expanding and transitive (resp.  ,   $(X,  f)$ is
a homeomorphism that is topologically hyperbolic and transitive).     Then $h_{top}(f)>0$ and
\begin{description}

\item[(I)] $\htop(\{x\in X\,  |\,  x \text{ satisfies Case } (i) \}\cap NRec(f)\cap \mathbb{E}_k\cap QR(f))=\htop(f),\, i=1,2,\,k=1,2,3,\cdots.  $


 \item[(II)] $\htop(\{x\in X\,  |\,   x \text{ satisfies Case } (1) \}\cap Rec(f)\cap   \mathbb{E}_k\cap QR(f))=\htop(f), k=1,2,3,\cdots.  $

%


   \end{description}

\end{maintheorem}

 \begin{Rem} For the set $\{x\in X\,  |\, x \text{ satisfies Case } (2) \}\cap Rec(f)$, it is nonempty in full shifts over finite symbols. Let us explain why as follows:
  We learned a example from \cite{HZ}  that  there is  a topological mixing but uniquely ergodic subshift $\Lambda$ of
 two symbols for which the unique invariant measure  is supported on a fixed point so that in this example $\emptyset
  \neq Tran(f|_\Lambda)\subseteq   Rec(f) \cap  \{x|\,\emptyset\neq \omega_{B_*}(x)=\omega_{\underline{d}}(x)=
  \omega_{\overline{d}}(x)= \omega_{B^*}(x)\subsetneq \omega_f(x)\},$  where $Tran(f|_\Lambda)$ denotes the set of all transitive points of subsystem $f|_\Lambda.$  However, we do not know the set of $\{x\in X\,  |\, x \text{ satisfies Case } (2) \}\cap Rec(f)$ has full toplogical entropy or not, though  this set is nonempty and  any full shift over finite symbols
  satisfies the assumption of Theorem \ref{thm-A0000000000}.
\end{Rem}

 \begin{Rem}\label{Rem-positiveentropy} It was proved in \cite[Corollary 4]{Moo2011} that if a dynamical
system with the shadowing property has a recurrent not minimal point, or a minimal
sensitive point, then the entropy must be positive. Thus, for a transitive dynamical system $(X,f)$ with shadowing property, if $h_{top}(f)=0,$ then $f$ is minimal and equicontinuous (implying unique ergodicity, see \cite{Brown}) so that all points in $X$ belong to Case (1). If further $(X,  f)$ is expansive, then there must exist periodic points so that $X$ is composed of a periodic orbit. This is why $h_{top}(f)>0$ under the assumptions of Theorem  \ref{thm-A0000000000}.
\end{Rem}

\subsection{Full Topological Entropy of Almost Periodic Points}


Periodic points and minimal points (or called almost periodic points) are two classical and basic concept in the study of dynamical systems.
    For periodic points, it is well-known that  the exponential growth of the periodic points 
      equals to topological entropy for hyperbolic systems but Kaloshin showed that
in general periodic points can grow much faster than entropy \cite{Kal}. Moreover,
it is well known that for $C^1$ generic diffeomorphisms, all periodic points are hyperbolic so that countable and they form a
 dense subset of the non-wandering set (by  classical  Kupka-Smale theorem, Pugh's  or Ma\~{n}$\acute{\text{e}}$'s
 Ergodic Closing lemma from   Smooth Ergodic Theory, for example,  see  \cite{Kupka1,Smale1,Pugh1,Pugh2,Mane1}).
 However, periodic point does not exist naturally. For example there is no periodic points in any irrational rotation. Almost periodic point is a good generalization which exists naturally  since it is equivalent that it belongs to a minimal set(
 see  \cite{Birkhoff,Gottschalk44-2222,Gottschalk44,Gottschalk46,Mai}) and
  by Zorn's lemma  any dynamical system contains at least one
minimal invariant subset.    There are many examples of subshifts which are strictly ergodic (so that every point in the subshift is almost periodic)
 and has positive entropy, for example, see \cite{Grillenberger}. In other words, almost periodic points have strong dynamical complexity.
 For weakly almost periodic points, quasi-weakly almost periodic points and Banach-recurrent points, it is shown recently that all their gap-sets with same asymptotic behavior carry high dynamical complexity in the sense of full topological entropy over dynamical system with specification-like property and expansiveness\cite{HTW,T16}.

  \begin{Def}

      A point $x\in X$ is {\it almost periodic} or {\it minimal},   if for every open neighborhood $U$ of $x$,   there exists $N\in\N$ such that $f^k (x)\in U$ for some $k\in [n,  n+N]$ and every $n\in\N.  $ Equivalently,
       $x\in X$ is {\it almost periodic} means that $x$ is recurrent and $\omega_f(x)$ is minimal.
 \end{Def}

Denote the set of almost periodic points by $AP(f)$. It is easy to check that
 $$x\in AP(f) \Leftrightarrow x\in \omega_{B_*}(x) \Leftrightarrow x\text{ satisfies Case (1) and }x\in Rec(f). $$
It is well-known that minimal subsystem always exists so that $AP(f)$ is nonempty and thus Case (1) always appear in any dynamical system.

 By item (II)  of Theorem \ref{thm-A0000000000}  $AP(f)$ carries full topological entropy when the system is transitive and expansive with shadowing property.
We point out that shadowing is enough to get full topological entropy of $AP(f).$

\begin{maintheorem}\label{thm-AABBBB}
  Suppose that $(X,  f)$ has shadowing property or  almost specification. 
   Then
   $\htop( AP(f) )=\htop( AP(f)\cap QR(f))=\htop( AP(f)\cap QR(f)\cap \cup_{k\geq l} \mathbb{E}_k )=\htop(f),\,l=1,2,3,\cdots.  $

\end{maintheorem}

 \begin{Rem} From \cite{T16} we know the set of periodic points (denoted by $Per(f)$)  carries zero topological entropy in any dynamical system and from \cite[Theorem 1.10]{T16} $AP(f)\setminus Per(f)$ carries full topological entropy for full shifts of finite symbols. By
 Theorem \ref{thm-AABBBB}   $AP(f)\setminus Per(f)$ carries full topological entropy for any dynamical system with shadowing property or  almost specification  so that Theorem \ref{thm-AABBBB} can be as a generalization of \cite[Theorem 1.10]{T16} from full shift to a large class of general dynamical systems as follows:\\
 (1)  subshifts of finite type;\\
 (2)   $\beta$-shifts; \\
 (3) mixing interval maps.\\
  This is because from \cite{Walters2} we know a subshift is finite type if and only if it has shadowing property; from \cite{PS2007} any $\beta$-shift has $g$-product property which is litte stronger than almsot specification\cite{Thompson2012}; and it is well-known that any mixing interval map has specification property which is stronger than $g$-product property.
\end{Rem}

\begin{Rem} It was proved in \cite[Corollary 4]{Moo2011} that if a dynamical
system with the shadowing property has a recurrent not minimal point, or a minimal
sensitive point, then the entropy must be positive. Thus, for a  dynamical system $(X,f)$ with shadowing property, if $h_{top}(f)=0,$ then $Rec(f)=AP(f)$ and  for any $x\in X,$ $\omega_{B^*}(x)$ must be minimal and equicontinuous (implying unique ergodicity in $\omega_{B^*}(x)$ and thus also in $\omega_f(x)$, see \cite{Brown}) so that $X= \cup_{\mu\in M_{erg}(f,X)} G_\mu=QR(f)$ and  any point in $X$ belongs to Case (1) or Case (2). It is little different from Remark \ref{Rem-positiveentropy} since here $f$ is not assumed transitive.  
\end{Rem}

Let    $C(M) $be the set of continuous maps on the space $ M$ and
$H(M)$ the set of homeomorphisms on $M$.    Recall that $C^0$ generic $f\in H(M)$ (or $f\in C(M)$) has the shadowing property (see \cite{KMO2014} and \cite{ Koscielniak2,  Koscielniak},   respectively) so that by  Theorem \ref{thm-AABBBB}  one has
\begin{maincorollary}\label{Corollary-A}
Let $M$ be a compact topological manifold (with or without boundary) of dimension at least $2$ and assume that $M$ admits a decomposition.   Then there is a residual subset $\mathcal {R}\subseteq H(M)$ (or $ \mathcal {R}\subseteq  C(M)$) such that for any $f\in \mathcal {R}$, $$\htop(AP(f))=\htop( AP(f)\cap QR(f) )=\htop( AP(f)\cap QR(f)\cap \cup_{k\geq l} \mathbb{E}_k )=\htop(f),\,l=1,2,\cdots.$$
\end{maincorollary}







 \subsection{Combination of Multifractal Analysis and Nonrecurrence}


{




For a continuous function $\varphi$ on $X$,   define the \emph{$\varphi-$irregular set} as
\begin{eqnarray*}
  I_{\varphi}(f) &:=& \left\{x\in X: \lim_{n\to\infty}\frac1n\sum_{i=0}^{n-1}\varphi(f^ix) \,\, \text{ diverges }\right\}.
\end{eqnarray*}
 Define
 $L_\varphi:=\left[\varphi_*,  \,  \varphi^*\right]
~\textrm{and}~
\interior L_\varphi:=\left(\varphi_*,\varphi^*\right)\text{ where } \varphi_*=\inf_{\mu\in M(f,  X)}\int\varphi d\mu \text{ and } \varphi^*=\sup_{\mu\in M(f,  X)}\int\varphi d\mu.$
 Let $IR(f):=\cup_{\phi\in C^0(X)} I_\phi(f)\,(=X\setminus QR(f))$ where $C^0(X)$ denotes the space of all continuous functions.  The points  are called points with historic behavior, 
  see \cite{Ruelle,Takens}. The set of points with historic behavior is also called  irregular set, denoted by $IR(f),$ see \cite{Pesin-Pitskel1984,Barreira-Schmeling2000,BPS,  Pesin,  FFW,CTS,Thompson2012,DOT}.
 From the viewpoint of ergodic theory,   the irregular points are negligible by Birkhoff Ergodic Theorem.   However,  they describe the points with same asymptotic behavior in the sense of ergodic average divergence.   Pesin and Pitskel \cite{Pesin-Pitskel1984} are the first to notice the phenomenon of the irregular set
carrying full topological entropy in the case of the full shift on two symbols from  There are lots of advanced results to show that the irregular points can carry full entropy in symbolic systems, hyperbolic systems, non-uniformly expanding or hyperbolic systems,  and systems with specification-like or shadowing-like properties, for example, see \cite{Barreira-Schmeling2000,BPS,  Pesin,  FFW,CTS,Thompson2012,DOT,LLST,TianVarandas}.

{ \color{red}  }

\begin{maintheorem}\label{thm-CCCCCCCCCCCCC}
   Suppose $(X,  f)$ is topologically expanding and transitive (resp.  ,   $(X,  f)$ is a homeomorphism that is topologically hyperbolic and transitive).
Let $\varphi$ be a continuous function on $X$ with $Int(L_\varphi)\neq \emptyset$.   Then
$\htop(\{x\in X\,  |\,  x\text{ satisfies Case (i) } \}\cap \mathbb{E}_k\cap NRec(f)\cap I_{\varphi}(f))=\htop(f)>0,\, i=1, 2,\,k=2,3,4,\cdots.  $




\end{maintheorem}

\begin{maincorollary}\label{Corollary-thm-CCCCCCCCCCCCC}
   Suppose $(X,  f)$ is topologically expanding and transitive (resp.  ,   $(X,  f)$ is a homeomorphism that is topologically hyperbolic and transitive).
    Then
$\htop(\{x\in X\,  |\,  x\text{ satisfies Case (i) } \}\cap \mathbb{E}_k\cap NRec(f)\cap IR(f))=\htop(f)>0,\, i=1, 2,\,k=2,3,4,\cdots.  $




\end{maincorollary}


For a continuous function $\varphi$ on $X$ and any $a\in  L_\varphi,  $ denote
$$t_a=\sup_{\mu\in M(f,  X)}\left\{h_\mu:\,  \int\varphi d\mu=a\right\}$$
and consider the level set
\begin{eqnarray*}
  R_{\varphi}(a) &:=& \left\{x\in X: \lim_{n\to\infty}\frac1n\sum_{i=0}^{n-1}\varphi(f^ix)=a\right\}.
\end{eqnarray*}
 Barreira and Saussol proved in \cite{Barreira---Saussol2001} (or see \cite{Bar2008}) the following properties for a dynamical system $(X,f)$ whose function of metric entropy is upper semi-continuous.  Consider a H$\ddot{\text{o}}$lder continuous function $\phi$ (see~\cite{Bar2011,BD} for almost additive functions with tempered variation) which has a unique equilibrium measure, then for any constant $a\in\Int (L_\phi)$ 
 \begin{equation}\label{eq-1} h_{top}(R_\varphi(a)=t_a.
 \end{equation}
 For $\phi$ being an arbitrary continuous function (hence there may exist more than one equilibrium measures), \eqref{eq-1} was established by Takens and Verbitski \cite{Takens-Verbitskiy} under the assumption that $f$
has the specification property. This result was further generalized by Pfister and Sullivan  \cite{PS2007} to dynamical systems with $g$-product property, see \cite{FengH,Thompson2009,T16} for more related discussions.  The  method used in \cite{BD,Barreira---Saussol2001} mainly depends on  thermodynamic formalism such as differentiability of
 pressure function while the method in \cite{Takens-Verbitskiy,PS2007} is a direct approach by constructing fractal sets.

\begin{maintheorem}\label{thm-DDDDDDDDDD} Suppose $(X,  f)$ is topologically expanding and transitive (resp.  ,   $(X,  f)$ is a homeomorphism that is topologically hyperbolic and transitive).
For a continuous function $\varphi$ on $X$ with $Int(L_\varphi)\neq \emptyset$ and $a\in Int(L_\varphi)$,     we have the following conditional variational principle: \\
 (I) $\htop(\{x\in X\,  |\,  x\text{ satisfies Case (i) } \}\cap \mathbb{E}_k\cap NRec(f)\cap R_{\varphi}(a)\cap QR(f))=t_a>0,\, i=1, 2,\,k=2,3,4,\cdots;  $ \\
  (II) $\htop(\{x\in X\,  |\,  x\text{ satisfies Case (i) } \}\cap \mathbb{E}_k\cap NRec(f)\cap R_{\varphi}(a)\cap IR(f))=t_a>0,\, i=1, 2,\,k=3,4,5,\cdots.  $ \\
   Moreover, for any $k=2,3,4,\cdots$,
 $$t_a=\sup\left\{h_{\mu}~|~\int\varphi d\mu=a,  \mu\in M(f,  X),  ~S_{\mu}\neq X,  ~S_\mu~\text{is minimal}~\text{and}~ \#M_{erg}(f,  S_\mu)=k\right\}.$$

\end{maintheorem}

\begin{Rem} It is still unknown whether Theorem \ref{thm-DDDDDDDDDD} item (I) can  hold also for $k=1$ and Theorem \ref{thm-DDDDDDDDDD} item (II) can  hold also for $k=2.$

\end{Rem}

Define $\phi-$regular set $R_\varphi(f):=\cup_{a\in \mathbb{R}} R_\varphi(a) \,(=X\setminus I_\varphi(f)).$

\begin{maincorollary}\label{Cor--thm-DDDDDDDDDD} Suppose $(X,  f)$ is topologically expanding and transitive (resp.  ,   $(X,  f)$ is a homeomorphism that is topologically hyperbolic and transitive).
For a continuous function $\varphi$ on $X$,     we have  \\
(I).
 $\htop(\{x\in X\,  |\,  x\text{ satisfies Case (i) } \}\cap \mathbb{E}_k\cap NRec(f)\cap R_{\varphi}(f)\cap QR(f))=\htop(f)>0,\, i=1, 2,\,k=1,2,3,4,\cdots;$\\
 (II).
 $\htop(\{x\in X\,  |\,  x\text{ satisfies Case (i) } \}\cap \mathbb{E}_k\cap NRec(f)\cap R_{\varphi}(f)\cap IR(f))=\htop(f)>0,\, i=1, 2,\,k=3,4,5,\cdots; $ \\
  (III) $\htop(  \mathbb{E}_k\cap AP(f)\cap R_{\varphi}(f)\cap QR(f))=\htop(f)>0,\,k=1,2,3,\cdots.$

\end{maincorollary}

{}

It is still unknown that
\begin{Que}  Whether it can be stated in Theorem  \ref{thm-CCCCCCCCCCCCC} and Corollary \ref{Corollary-thm-CCCCCCCCCCCCC}  that
 $$\htop(A(f)\cap I_{\varphi}(f))=\htop(A(f)\cap IR(f))=\htop(f)>0,  $$ and
 whether it can be stated in Theorem  \ref{thm-DDDDDDDDDD} and Corollary \ref{Cor--thm-DDDDDDDDDD} that  $$\htop(A(f)\cap R_{\varphi}(a))=\htop(A(f)\cap R_{\varphi}(a)\cap QR(f))=\htop(A(f)\cap R_{\varphi}(a)\cap IR(f))=t_a>0  $$ and $\htop(    AP(f)\cap R_{\varphi}(f)\cap IR(f))=\htop(f)>0?$

\end{Que}

This question is meaningful by following theorem.

\begin{maintheorem}\label{theorem-minimal-irregular}
   Suppose $(X,  f)$ is topologically expanding and transitive (resp.  ,   $(X,  f)$ is a homeomorphism that is topologically hyperbolic and transitive).
Let $\varphi$ be a continuous function on $X$ with $Int(L_\varphi)\neq \emptyset$.   Then
 $A(f)\cap I_{\varphi}(f)\cap \mathbb{E}_k$ is nonempty   in $X,$  where $\,k=2,3,4,\cdots.  $

\end{maintheorem}



\subsection{Examples}
\subsubsection{Symbolic dynamics}

All above results are suitable for all transitive subshift of finite type because it is naturally expansive
and from \cite{Walters2} a subshift satisfies shadowing property if and only if it is a subshift of finite type.

\subsubsection{Smooth dynamics}
We now suppose that $f:M\to M$ is a diffeomorphism of a compact $C^{\infty}$ Riemannian manifold $M$.   Then the derivative of $f$ can be considered a map $df: TM\to TM$ where $TM=\bigcup_{x\in M}T_xM$ is the tangent bundle of $M$ and $df_x:T_xM\to T_{f(x)}M$.
A closed subset $\Lambda\subset M$ is \emph{hyperbolic} if $f(\Lambda)=\Lambda$ and each tangent space $T_xM$ with $x\in\Lambda$ can be written as a direct sum
 $T_xM=E_x^u\oplus E_x^s$
of subspaces so that
\begin{enumerate}
  \item $Df(E_x^s)=E_{f(x)}^s,  Df(E_x^u)=E_{f(x)}^u$;
  \item there exist constants $c>0$ and $\lambda\in(0,  1)$ so that
  $$\|Df^n(v)\|\leq c\lambda^n\|v\|~\textrm{when}~v\in E_x^s,  n\geq 1,\, \text{   and } \,\|Df^{-n}(v)\|\leq c\lambda^n\|v\|~\textrm{when}~v\in E_x^u,  n\geq 1.$$
\end{enumerate}
Furthermore,   we say $f$ satisfies \emph{Axiom A} if $\Omega(f)$ is hyperbolic and the periodic points are dense in $\Omega(f).  $
It is well known that the basic set of Axiom A systems is expansive,   transitive and satisfies  the shadowing property so that all above results
  are suitable for the subsystem restricted on basic set. In particular, all above results
  are suitable for all transitive Anosov dfiffeomorphisms (for which the whole space $M$ is hyperbolic).

\subsubsection{Nonuniformly Hyperbolic Dynamical Systems}

It is known that uniform hyperbolicity is not a dense property in the space of differential dynamical systems so that Pesin and Katok etc. introduced a more general concept in the probabilistic perspective, called nonuniform hyperbolicity, which plays important roles in the study of modern dynamics, see \cite{BP,Katok}.
Let $f$ be a $C^{1+\alpha}$ diffeomorphism over a compact Riemannian manifold $M.$
An ergodic measure is called {\it hyperbolic} if its all Lyapunov exponents are non-zero. It was proved in \cite{Katok,Katok1} that the metric entropy of a hyperbolic ergodic measure can be approximated by the topological entropy of a transitive topological hyperbolic set (called horseshoe there). Thus by  Corollary \ref{Corollary-thm-CCCCCCCCCCCCC} and Theorem \ref{thm-A0000000000} one has

\begin{maincorollary}\label{Corollary-B-nonuniformlyhyperbolic}
Let $M$ be a compact Riemannian manifold  of dimension at least $2$ and $f$ be a $C^{1+\alpha}$ diffeomorphism.   Then
\\(I).   $\htop(\{x\in X\,  |\,  x\text{ satisfies Case (i) } \}\cap \mathbb{E}_k\cap NRec(f)\cap IR(f))\geq h_{hyp},\, i=1, 2,\,k=2,3,4,\cdots;  $
\\(II).   $\htop(\{x\in X\,  |\,  x\text{ satisfies Case (i) } \}\cap \mathbb{E}_k\cap NRec(f)\cap QR(f))\geq h_{hyp},\, i=1, 2,\,k=1,2,3,\cdots;$ and
\\(III).$  \htop( AP(f)\cap QR(f) )\geq h_{hyp},$
where $$h_{hyp}:=\sup\{\htop(f|_\Lambda)|\, \Lambda\text{ is a transitive topological hyperbolic set}\}$$$$=\sup\{h_\mu(f)|\, \mu\text{ is a hyperbolic ergodic measure}\}.$$
\end{maincorollary}

\begin{Rem} This corollary also can be stated for $C^1$ diffeomorphisms with dominated splitting by replacing Katok's horseshoe lemma by the  one in \cite{SunTian-C1Pesinset,STV}. 

\end{Rem}

For a surface diffeomorphism,  any ergodic measure with positive metric entropy should be hyperbolic by classical  Ruelle's inequality on metric entropy and Lyapunov exponents \cite{Ru}. In other words,  any  surface diffeomorphism $f$ satisfies that  $$ \htop(f)=\sup\{h_\mu(f)|\, \mu\text{ is     ergodic  }\}=\sup\{h_\mu(f)|\, \mu\text{ is   hyperbolic and ergodic}\}$$ and thus we have following.

\begin{maincorollary}\label{Corollary-B-nonuniformlyhyperbolic}
Let  $f$ be a $C^{1+\alpha}$ surface diffeomorphism.   Then
\\(I). $\htop(\{x\in X\,  |\,  x\text{ satisfies Case (i) } \}\cap \mathbb{E}_k\cap NRec(f)\cap IR(f))  =\htop(f),\, i=1, 2,\,k=2,3,4,\cdots;  $
\\(II).  $\htop(\{x\in X\,  |\,  x\text{ satisfies Case (i) } \}\cap \mathbb{E}_k\cap NRec(f)\cap QR(f))=\htop(f),\, i=1, 2,\,k=1,2,3,\cdots;$  and
\\(III).$  \htop( AP(f)\cap QR(f)\cap \mathbb{E}_k )=\htop(f), k=1,2,3,\cdots.$
\end{maincorollary}

\subsection{Organization of the paper}

The remainder of this paper is organized as follows.  In Section \ref{section-Perliminaries}
we will recall the notions of entropy, shadowing and shadowing-like properties, saturated property.
 In   Section \ref{section-minimalentropydense} we show minimal-entropy-dense property for which the proof for symbolic version is given in the last section.  In Section \ref{Proof} we give the proofs of our main results.


\section{{  Preliminaries}}\label{section-Perliminaries}

Consider a metric space $(O,  d).  $ Let $A,  B$ be two nonempty subsets,   then the distance from $x\in X$ to $B$ is defined as
$$\dist(x,  A):=\inf_{y\in B}d(x,  y).  $$
Furthermore,   the distance from $A$ to $B$ is defined as
$$\dist(A,  B):=\sup_{x\in A}\dist(x,  B).  $$
Finally,   the Hausdorff distance between $A$ and $B$ is defined as
$$d_H(A,  B):=\max\set{\dist(A,  B),  \dist(B,  A)}.  $$
 Now consider a TDS $(X,  f).  $ If for every pair of non-empty open sets $U,  V$ there is an integer $n$ such that $f^n(U)\cap V\neq \emptyset$ then we call $(X,  f)$ \textit{topologically transitive}.
Furthermore,   if for every pair of non-empty open sets $U,  V$ there exists an integer $N$ such that $f^n(U)\cap V\neq \emptyset$ for every $n>N$,   then we call $(X,  f)$ \textit{topologically mixing}.
We say that $f$ is \emph{positively expansive} if there exists a constant $c>0$ such that for any $x,  y\in X$,   $d(f^ix,  f^iy)> c$ for some $i\in\Z^+$.   When $f$ is a homeomorphism,   we say that $f$ is \emph{expansive} if there exists a constant $c>0$ such that for any $x,  y\in X$,   $d(f^ix,  f^iy)> c$ for some $i\in\Z$.   We call $c$ the expansive constant.
We say $(Y,  f|_Y)$ is a subsystem of $(X,  f)$ if $Y$ is a closed $f$-invariant subset of $X$ and $f|_Y$ is the restriction of $f$ on $Y$.   It is not hard to check that $Rec(f|_Y)=Rec(f)\cap Y$.   Consequently,   $NRec(f|_Y)=NRec(f)\cap Y$.

A finite sequence $\C=\langle x_1,  \cdots,  x_l\rangle,  l\in\N$ is called a \emph{chain}.   Furthermore,   if $d(fx_i,  x_{i+1})<\eps,  1\leq i\leq l-1$,   we call $\C$ an $\eps$-chain.
Let $A\subseteq X$ be a nonempty invariant set.   We call $A$ \emph{internally chain transitive} if for any $a,  b\in A$ and any $\eps>0$,   there is an $\eps$-chain $\mathfrak{C}_{ab}$ in $A$ with connecting $a$ and $b$.

\begin{Lem}\label{Lem-omega-naturally-in-ICT} \cite{HSZ} For any $x\in X,$ $\omega_f(x)$ is chain transitive.

\end{Lem}

Let $\omega_f=\{A\subseteq X:\exists x\in X~\text{with}~\omega_f(x)=A\}$ and denote the collection of internally chain transitive sets by $ICT(f)$.
For any open $U\subseteq X,  $ define
$$\omega_f^U:=\{A\subseteq X: \text{there exists } x\in X \text{ such that } \omega_f(x)=A\}.  $$

\begin{Lem} \label{Lem-locally-omega-f-equalto-ICT}  \cite{DongTian2016-nosyndetic}
Suppose $(X,  f)$ is topologically transitive and topologically expanding. 
  Then  for any nonempty open $U\subseteq X,  $ $ICT(f)=\omega_f=\omega^U_f$.
\end{Lem}

For any two TDSs $(X,  f)$ and $(Y,  g)$,   if $\pi:(X,  f)\to (Y,  g)$ is a continuous surjection such that $\pi\circ f=g\circ\pi$,   then we say $\pi$ is a semiconjugation.   We have the following conclusions:
\begin{equation}\label{AP-Rec}
 \pi(AP(f))=AP(g)~\text{and}~\pi(Rec(f))=Rec(g).
\end{equation}
If in addition,   $\pi$ is a homeomorphism,   we call $\pi$ a conjugation.

\subsection{The Space of Borel Probability Measures}
\subsubsection{Probability Measures and Invariant Measures}

The space of Borel probability measures on $X$ is denoted by $M(X)$ and the set of continuous functions on $X$ by $C(X)$.   We endow $\varphi\in C(X)$ the norm $\|\varphi\|=\max\{|\varphi(x)|:x\in X\}$.
Let ${\{\varphi_{j}\}}_{j\in\mathbb{N}}$ be a dense subset of $C(X)$,   then
$$\rho(\xi,  \tau)=\sum_{j=1}^{\infty}\frac{|\int\varphi_{j}d\xi-\int\varphi_{j}d\tau|}{2^{j}\|\varphi_{j}\|}$$
defines a metric on $M(X)$ for the $weak^{*}$ topology  \cite{Walters}.
For $\nu\in M(X)$ and $r>0$,   we denote a ball in $M(X)$ centered at $\nu$ with radius $r$ by
$$\mathcal{B}(\nu,  r):=\{\rho(\nu,  \mu)<r:\mu\in M(X)\}.  $$
\begin{Def}
For $\mu\in M(X)$,   the set of all $x\in X$ with the property that $\mu(U)>0$ for all neighborhood $U$ of $x$ is called the support of $\mu$ and denoted by $S_{\mu}$.   Alternatively,   $S_{\mu}$ is the (well defined) smallest closed set $C$ with $\mu(C)=1$.
\end{Def}

We say $\mu\in M(X)$ is an $f$-invariant measure if for any Borel measurable set $A$,   one has $\mu(A)=\mu(f^{-1}A)$.   The set of $f$-invariant measures are denoted by $M(f,  X)$.   We remark that if $\mu\in M(f,  X)$,   then the support of $\mu$ $S_{\mu}$ is a closed $f$-invariant set.

{
Let $(X,  f)$ and $(Y,  g)$ be two dynamical systems and $\pi:X\to Y$ is a continuous map.   Define $\pi_*:M(X)\to M(Y)$ as $\pi_*\mu:=\mu\circ\pi^{-1}$ and $\pi^*:C(Y)\to C(X)$ as $\pi^*\varphi:=\varphi\circ\pi.  $ We call $\pi_*\mu$ the push-forward of $\mu$ and $\pi^*\varphi$ the pull-back of $\varphi$.   It is not hard to see that
\begin{equation}\label{eq-pi-up-down}
  \int\varphi d\pi_*\mu=\int\pi^*\varphi d\mu~\textrm{for any}~\varphi\in C(Y)~\textrm{and}~\mu\in M(X).
\end{equation}
Moreover,   one notices that $\mu\in M(f,  X)$ is equivalent to $f_*\mu=\mu.  $ The following are not hard to check.
\begin{Lem}\label{lem-pi-star}
If $\pi:(X,  f)\to (Y,  g)$ is a conjugation,
then
\begin{description}
  \item[(1)] there is a metric $\rho_X$ on $M(X)$ and a metric $\rho_Y$ on $M(Y)$ such that $\pi_*:(M(X),  \rho_X)\to(M(Y),  \rho_Y)$ is an isometric isomorphism.   In particular,   $\pi_*:M(X)\to M(Y)$ is a homeomorphism.
  \item[(2)] $\pi_*(M(f,  X))= M(g,  Y)$ and $\pi_*|_{M(f,  X)}:(M(f,  X),  \rho_X)\to(M(f,  Y),  \rho_Y)$ is also an isometry isomorphism.
\end{description}
\end{Lem}
We say $\mu\in M(X)$ is an ergodic measure if for any Borel set $B$ with $f^{-1}B=B$,   either $\mu(B)=0$ or $\mu(B)=1$.   We denote the set of ergodic measures on $X$ by $M_{erg}(f,  X)$.   It is well known that the ergodic measures are exactly the extreme points of $M(f,  X).  $
The following is an easy observation.

\begin{Lem}\label{lem-decomposition}
If $\mu\in M_{erg}(f,  X)$,   then for any $n\in\N$,   there is a $\nu\in M_{erg}(f^n,  X)$ and
an $m\in\N$ with $m|n$ such that $\mu$ can decompose as
$$\mu=\frac1m(\nu+f_*\nu+\cdots+f_*^{m-1}\nu).  $$
Moreover,   there is a $X_0\subseteq X$ with $\nu(X_0)=1$ and $f^mX_0=X_0$ such that $X$ has a $\mod 0$ measurable partition $X=\bigsqcup_{i=0}^{m-1}f^iX_0.  $
\end{Lem}

\begin{Lem}\label{lem-Lambda-Lambda-0} \cite{DongTian2016-nosyndetic}
Let $\Lambda,  \Lambda_0$ be three closed $f$-invariant subset of $X$ with $\Lambda_0\subseteq \Lambda$.   Suppose for any $x\in\Lambda$,   $\omega_f(x)\subseteq \Lambda_0$.   Then
$$M(f,  \Lambda)=M(f,  \Lambda_0).  $$
\end{Lem}

}
{

}

 The set of $f$-invariant measures and ergodic measures supported on $Y\subseteq X$ are denoted by $M(f,  Y)$ and $M_{erg}(f,  Y)$ respectively.
     Let $S_\mu$ denote the support of $\mu.$
For $x\in X$,   we define the empirical measure of $x$ as
\begin{equation*}
 \mathcal{E}_{n}(x):=\frac{1}{n}\sum_{j=0}^{n-1}\delta_{f^{j}(x)},
\end{equation*}
where $\delta_{x}$ is the Dirac mass at $x$.    We denote the set of limit points of $\{\mathcal{E}_{n}(x)\}$ by $V_{f}(x)$.   As is known,   $V_f(X)$ is  a non-empty compact connected subset of $M(f,  X)$ \cite{DGS}.
     { For any two positive integers $a_k<b_k$,   denote
$[a_k,  b_k]=\{a_k,  a_k+1,  \cdots,  b_k\}$ and $[a_k,  b_k)=[a_k,  b_k-1],  (a_k,  b_k)=[a_k+1,  b_k-1],  (a_k,  b_k]=[a_k+1,  b_k]$.
A point $x$ is called {\it quasi-generic} for some measure $\mu,  $ if there is a sequence of positive integer intervals $I_k=[a_k,  b_k)$ with $b_k-a_k\to\infty$ such that
$$\lim_{k\rightarrow\infty}\frac{1}{b_k-a_k}\sum_{j=a_k}^{b_k-1}\delta_{f^j(x)}=\mu$$
in weak$^*$ topology.
}
Let $V_f^*(x)=\{\mu\in M(f,  X): \,  x \text{ is quasi-generic for } \mu\}$.
 This concept is from \cite{Furst} and from there it is known $V_f^*(x)$ is always nonempty,
 compact and connected.   Note that $V_f(x)\subseteq V_f^*(x).  $


\begin{Thm}\label{thm-density-basic-property} \cite{DongTian2016-nosyndetic} Suppose $(X,  f)$ is a topological dynamical system.
\begin{description}
  \item[(1)] For any $x\in X,   $ $\omega_{\underline{d}}(x)= \bigcap_{\mu\in V_f(x)} S_\mu$.
  \item[(2)] For any $x\in X,   $  $\omega_{\overline{d}}(x)=\overline{\bigcup_{\mu\in V_f(x)} S_\mu}\neq \emptyset$.
  \item[(3)] For any $x\in X,   $  $\omega_{\underline{B}}(x)= \bigcap_{\mu\in V^*_f(x)} S_\mu =  {\bigcap_{\mu\in M(f,  \omega_f(x))}S_{\mu}}= {\bigcap_{\mu\in M_{erg}(f,  \omega_f(x))}S_{\mu}}.  $\,  \,  \,   If \,
   $\omega_{\underline{B}}(x)\neq \emptyset$,   then $ \omega_{\underline{B}}(x)$ is minimal.
  \item[(4)] For any $x\in X,   $ $\omega_{\overline{B}}(x)= \overline{\bigcup_{\mu\in V^*_f(x)} S_\mu}=\overline{\bigcup_{\mu\in M(f,  \omega_f(x))}S_{\mu}}=\overline{\bigcup_{\mu\in M_{erg}(f,  \omega_f(x))}S_{\mu}}\neq \emptyset;$

\end{description}

 \end{Thm}

\subsection{Topological Entropy and Metric Entropy}\label{topological-entropy}

\subsubsection{Topological Entropy for Noncompact Set}
As for noncompact sets ,
Bowen also developed a satisfying definition via dimension language \cite{Bowen1973} which we now illustrate.
Let $E\subseteq X$,   and $\mathcal {G}_{n}(E,  \sigma)$ be the collection of all finite or countable covers of $E$ by sets of the form $B_{u}(x,  \sigma)$ with $u\geq n$.   We set
$$C(E;t,  n,  \sigma,  f):=\inf_{\mathcal {C}\in \mathcal {G}_{n}(E,  \sigma)}\sum_{B_{u}(x,  \sigma)\in \mathcal {C}}e^{-tu} \,\,\,\text{   and } C(E;t,  \sigma,  f):=\lim_{n\rightarrow\infty}C(E;t,  n,  \sigma,  f).  $$
Then we define
$$\htop(E;\sigma,  f):=\inf\{t:C(E;t,  \sigma,  f)=0\}=\sup\{t:C(E;t,  \sigma,  f)=\infty\}$$
The \textit{Bowen topological entropy} of $E$ is
\begin{equation}\label{definition-of-topological-entropy}
  \htop(f,  E):=\lim_{\sigma\rightarrow0} \htop(E;\sigma,  f).
\end{equation}

\begin{Lem}\label{lem-Bowen} \cite{Bowen1973}
Let $f:X\rightarrow X$ be a continuous map on a compact metric space.   Set
$$QR(t)=\{x\in X~|~\exists\tau\in V_{f}(x)~\textrm{with}~h_{\tau}(T)\leq t\}.  $$
Then
 $h_{top}(f,  QR(t))\leq t.  $
\end{Lem}
\subsubsection{Metric Entropy}
We call $(X,  \B,  \mu)$ a probability space if $\B$ is a Borel $\sigma-$algebra on $X$ and $\mu$ is a probability measure on $X$.   For a finite measurable partition $\xi=\{A_1,  \cdots,  A_n\}$ of a probability space $(X,  \B,  \mu)$,   define
 $H_\mu(\xi)=-\sum_{i=1}^n\mu(A_i)\log\mu(A_i).  $
Let $f:X\to X$ be a continuous map preserving $\mu$.   We denote by $\bigvee_{i=0}^{n-1}f^{-i}\xi$ the partition whose element is the set $\bigcap_{i=0}^{n-1}f^{-i}A_{j_i},  1\leq j_i\leq n$.   Then the following limit exists:
$$h_\mu(f,  \xi)=\lim_{n\to\infty}\frac1n H_\mu\left(\bigvee_{i=0}^{n-1}f^{-i}\xi\right)$$
and we define the metric entropy of $\mu$ as
$$h_{\mu}(f):=\sup\{h_\mu(f,  \xi):\xi~\textrm{is a finite measurable partition of X}\}.  $$



\subsection{Pseudo-orbit Tracing Properties}

Bowen \cite{Bowen1975} proved that every Anosov diffeomorphism of a compact manifold has the shadowing property and he used this notion efficiently in the study of  $\omega$-limit sets.   Shadowing property is typically satisfied by families of tent maps \cite{CKY} and subshift of finite type \cite{Kurka}.   Since its introduction,   shadowing property has attracted various attentions.   Moreover,   shadowing property is also satisfies for $C^0$ generic systems \cite{Pilyugin-Plamenevskaya,  KMO2014}.   Meanwhile,   shadowing property is also generalized to various other forms.   For example,   there are studies on limit-shadowing \cite{Pilyugin2007},   s-limit-shadowing \cite{Mazur-Oprocha,  Sakai2012,PilSakai},   average-shadowing \cite{Kwietniak-Oprocha},   asymptotic-average-shadowing \cite{DTY},   thick shadowing \cite{BMR},   $\underline{d}$-shadowing \cite{Dastjerdi-Hosseini,  ODH},   and ergodic shadowing \cite{Fakhari-Gane}.

\subsubsection{Shadowing Property }
\begin{Def} For any $\delta>0$,   a sequence $\{x_n\}_{n=0}^{+\infty}$ is called a \textit{$\delta$-pseudo-orbit} if
$$d(f(x_n),  x_{n+1})<\delta~~\textrm{for}~~n\in\Z^+.  $$
Furthermore,   $\{x_n\}_{n=0}^{+\infty}$ is \textit{$\eps$-shadowed} by some $y\in X$ if
$$d(f^n(y),  x_n)<\eps~~\textrm{for any}~~n\in\Z^+.  $$
Finally,   we say that $(X,  f)$ has the \textit{shadowing property} if for any $\eps>0$,   there exists  $\delta>0$ such that any $\delta$-pseudo-orbit is $\eps$-shadowed by some point in $X$.
\end{Def}

\begin{Lem}\label{lem-omega-A} \cite{DongTian2016-nosyndetic}
Suppose $(X,  f)$ is topologically expanding and transitive (resp.  ,   $(X,  f)$ is a homeomorphism that is topologically hyperbolic and transitive).   Let $A\subsetneq X$ be a closed invariant subset of $X$.   Then there exists an $x\in X$ such that
\begin{equation}\label{A-f-A}
 A\subseteq\omega_f(x)\subseteq \cup_{l=0}^{\infty}f^{-l}A.
\end{equation}
In particular,
\begin{equation}\label{omega-f-x-A}
 \overline{\bigcup_{y\in\omega_f(x)}\omega_f(y)}\subseteq A\subseteq\omega_f(x)\neq X.
\end{equation}

\end{Lem}

\subsubsection{Specification and  Almost Specification Properties}
\begin{Def}
We say that $(X,  f)$ satisfies the \textit{specification} property if for all $\eps>0$,    there exists an integer $m(\eps)$ such that for any collection $\{I_{j}=[a_{j},  b_{j}]\subseteq\mathbb{Z}^{+}:j=1,  \cdots,  k\}$ of finite intervals with $a_{j+1}-b_{j}\geq m(\eps)$ for $j=1,  \cdots,  k-1$ and any $x_{1},  \cdots,  x_{k}$ in $X$,   there exists a point $x\in X$ such that
$$d(f^{a_{j}+t}x,  f^{t}x_{j})<\eps $$ for all $t=0,  \cdots,  b_{j}-a_{j}$ and $j=1,  \cdots,  k$.
\end{Def}

Pfister and Sullivan generalized the specification property to the $g$-approximate product property in the study of large deviation \cite{PS2005}.   Later on,   Thompson  renamed it as the almost specification property in the study of irregular points \cite{Thompson2012}.   The only difference is that the blowup function $g$ can depend on $\eps$ in the latter case.   However,   this subtle difference does not affect our discussion here.

\begin{Def}
Let $\eps_0>0$.   A function $g:\mathbb{N}\times (0,  \eps_0)\to \mathbb{N}$ is called a \emph{mistake function} if for all $\eps\in (0,  \varepsilon_0)$ and all $n\in \mathbb{N}$,   $g(n,  \eps)\leq g(n+1,  \eps)$ and
$$\lim_{n}\frac{g(n,  \eps)}{n}=0.  $$
\end{Def}
If $\eps\geq \eps_0$,   we define $g(n,  \eps)=g(n,  \eps_0)$.
For $n\in \mathbb{N}$ large enough such that $g(n,  \eps)<n$,   let $\Lambda_n=\{0,  1,  \cdots,  n-1\}$.   Define
the $(g;n,  \eps)$-Bowen ball centered at $x$ as the closed set
$$B_n(g;x,  \eps):=\{y\in X~|~\exists~\Lambda\subseteq \Lambda_n,  ~|\Lambda_n\setminus\Lambda|\leq g(n,  \eps)~\text{and}~\max\{f^jx,  f^jy:j\in\Lambda\}\leq\eps\}.  $$
\begin{Def}
The dynamical system $(X,  f)$ has the \emph{almost specification} property with mistake function $g$,   if there exists a function $k_g:(0,  +\infty)\to \N$ such that for any
$\eps_{1}>0,  \cdots,  \eps_{m}>0$,   any points $x_{1},  \cdots,  x_{m}\in X$,   and any integers $n_{1}\geq k_g(\eps_{1}),  \cdots,  n_{m}\geq k_g(\eps_{m})$,   we can find a point $z\in X$ such that
\begin{equation*}
  f^{l_{j}}(z)\in B_{n_{j}}(g;x_{j},  \eps_{j}),  ~j=1,  \cdots,  m,
\end{equation*}
where $n_{0}=0~\textrm{and}~l_{j}=\sum_{s=0}^{j-1}n_{s}$.
\end{Def}

\subsection{{Saturated   and Entropy-dense  }}

In this section,   we consider a topological dynamical system $(X,  f)$.   For $x\in X$,   we denote the set of limit points of $\{\mathcal{E}_{n}(x)\}$ by $V_{f}(x)$.   As is known,   $V_f(X)$ is  a non-empty compact connected subset of $M(f,  X)$ \cite{PS2007}.   So for any non-empty compact connected subset $K$ of $M(f,  X)$,   it is logical to define the following set
$$G_K:=\{x\in X~|~V_f(X)=K\}.  $$
We call $G_K$ the saturated set of $K$.   Particularly,   if $K=\{\mu\}$ for some ergodic measure $\mu$,   then $G_{\mu}$ is just the generic points of $\mu$.   Saturated sets are studied by Pfister and Sullivan in \cite{PS2007}.
  When the dynamical system $f$ satisfies $g$-almost product and uniform separation property,   Pfister and Sullivan proved in \cite{PS2007} that $f$ is saturated and Huang,   Tian and Wang proved in \cite{HTW} that $f$ is transitively-saturated.

Let $\Lambda\subseteq X$ be a closed invariant subset and $K$ is a non-empty compact connected subset of $M(f,  \Lambda)$.   Define
$$G_K^{\Lambda}:=G_K\cap\{x\in X~|~\omega_f(x)=\Lambda\}   $$
 and
$$G_K^N:=G_K\cap NRec(f)~\textrm{and}~G_K^{\Lambda,  N}:=G_K^\Lambda\cap G_K^N.$$
\begin{Def}We say that $\Lambda\subsetneq X$ is nonrecurrently-star-saturated,   if for any non-empty connected compact set $K\subseteq M(f,  \Lambda)$,   one has
\begin{eqnarray} \label{eq-nonrecurrently-star-saturated-definition}
h_{top} (f,  G_K^{\Lambda,  N})=\inf\{h_\mu (f)\,  |\,  \mu\in K\}.
\end{eqnarray}
\end{Def}
\begin{Lem}\label{prop-locally-star-implies-nonrecurrently-star} \cite{DongTian2016-nosyndetic}
Suppose $(X,  f)$ is topologically transitive and topologically expanding (resp.,   a transitive and topologically hyperbolic homeomorphism). Let $\Lambda\subsetneq X$ be a closed $f$-invariant subset.
Then $\Lambda$ is nonrecurrently-star-saturated.

\end{Lem}

Now we recall the strong-basic-entropy-dense property from \cite{DongTian2016-nosyndetic}.
\begin{Def}\label{def-strong-basic}
We say $(X,  f)$ satisfies the strong-basic-entropy-dense property if for any $K=\cov\{\mu_i\}_{i=1}^m\subseteq M(f,  X)$ and any $\eta,  \zeta>0$,   there exist compact invariant subset $\Lambda_i\subseteq\Lambda\subsetneq X$,   $1\leq i\leq m$ such that
\begin{enumerate}
  \item $\Lambda$ is transitive and has the shadowing property.
  \item For each $1\leq i\leq m$,
  $\htop(f,  \Lambda_i)>h_{\mu_i}-\eta$ and consequently,   $\htop(f,  \Lambda)>\sup\{h_{\kappa}:\kappa\in K\}-\eta$.
  \item $d_H(K,  M(f,  \Lambda))<\zeta$,   $d_H(\mu_i,  M(f,  \Lambda_i))<\zeta$.
\end{enumerate}

\end{Def}

}
\begin{mainlemma}\label{Mainlemma-convex-by-horseshoe} \cite{DongTian2016-nosyndetic}
Suppose $(X,  f)$ is topologically expanding and transitive.   Then $(X,  f)$ satisfies the strong-basic-entropy-dense property.
\end{mainlemma}

\section{Minimal Entropy-dense Properties}\label{section-minimalentropydense}
Eizenberg,   Kifer and Weiss proved for systems with the specification property that \cite{EKW} any $f$-invariant probability measure $\nu$ is the weak limit of a sequence of ergodic measures
$\{\nu_n\}$,   such that the entropy of $\nu$ is the limit of the entropies of the $\{\nu_n\}$.   This is a central point in
large deviations theory,   which was first emphasized in \cite{FO}.   Meanwhile,   this also plays an crucial part in the computing of Billingsley dimension \cite{Billingsley1960,  Billingsley1961} on shift spaces \cite{PS2003}.
Pfister and Sullivan refer to this property as the \emph{entropy-dense} property \cite{PS2005}.
 In this subsection,   we introduce another entropy-dense property, called minimal-entropy-dense,   which shall serve for our main results.

{
\begin{Prop}\label{prop-f-f-k}
Consider $\Delta\subseteq X$ which is $f^k$-minimal for some $k\in\N$ and let $\Lambda=\bigcup_{i=0}^{k-1}f^i\Delta.  $ If we define $\overline{\nu}=j(\nu)=\frac1k\sum_{i=0}^{k-1}f_*^i\nu$ for each $\nu\in M(f^k,  \Delta)$,   then
\begin{description}
  \item[(1)] $j(M(f^k,  \Delta))\subseteq M(f,  \Lambda).  $
Moreover,   there is a metric $\rho_1$ on $M(f^k,  \Delta)$ and a metric $\rho_2$ on $M(f,  \Lambda)$ such that $j:(M(f^k,  \Delta),  \rho_1) \to (M(f,  \Lambda),  \rho_2)$ is an isometric imbedding.
  \item[(2)] Furthermore,
\begin{eqnarray*}
  j:(M_{erg}(f^k,  \Delta),  \rho_1) &\to& (M_{erg}(f,  \Lambda),  \rho_2) \\
  \nu &\mapsto& \overline{\nu}
\end{eqnarray*}
is an isometric isomorphism.
\item[(3)] $h_{\nu}(f^k|_\Delta)=kh_{\overline{\nu}}(f).  $
\end{description}

\end{Prop}

}

\begin{proof} Since $\Delta $   is $f^k$-minimal, we may assume that the union $\bigcup_{i=0}^{k-1}f^i\Delta  $ should be disjoint union. Otherwise, there must
 exist $1\leq m\leq k-1$ with $m | k$ such that $f^m\Delta=\Delta.$ Take $m$ be the smallest one, in this case $\Lambda=\bigcup_{i=0}^{m-1}f^i\Delta  $ should be disjoint union and then one can change $k$ to $m$ to give the proof. 

 (1) For any $\nu\in M(f^k,\Delta)$, $f_*^k\nu=(f^k)_*\nu=\nu$. So $f_*\overline{\nu}=\overline{\nu},$ which implies that $j(M(f^k,\Delta))\subseteq M(f,\Lambda).$ To show that $j$ is injective, let us assume $j(\mu)=j(\nu).$ Note that each $f^i\Delta$ is closed and $f_*^i\nu(f^i\Delta)=1$. So for each $0\leq i<j\leq k-1$,
      $$f_*^i\nu(f^i\Delta\cap f^j\Delta)=f_*^j\nu(f^i\Delta\cap f^j\Delta)=0.$$
      Therefore, for any Borel set $A\subseteq\Delta,$ one has $f_*^l\mu(A)=f_*^l\nu(A)=0,1\leq l\leq k-1.$ This implies that $\mu(A)=kj(\mu)(A)=kj(\nu)(A)=\nu(A).$ Since $A\subseteq\Delta$ is arbitrary, $\mu=\nu.$ Meanwhile, since $f_*$ is continuous, $j$ is continuous, which implies that $j$ is an imbedding. Finally, $M(f^k,\Delta)$ is compact, so the topology does not depend on the selection of metric on $M(f^k,\Delta)$. This enables us to pick the required $\rho_1$ and $\rho_2.$

  (2) Since $j$ is an isometry, it is sufficient to prove that $j:(M_{erg}(f^k,\Delta),\rho_1) \to (M_{erg}(f,\Lambda),\rho_2)$ is well-defined and is a surjection. Indeed, for any $A\subseteq\Lambda$ with $f^{-1}A=A$, one has $f^{-k}A=A$ and thus $f^{-k}(A\cap \Delta)=A\cap\Delta.$ Moreover, $\nu(A)=f_*\nu(A)=\cdots=f_*^{k-1}\nu(A)$. So if $0<j(\nu)(A)<1$, one has $0<\nu(A)<1,$ contradicting the fact that $\nu\in M(f^k,\Delta).$ This proves that $j(M_{erg}(f^k,\Delta))\subseteq M_{erg}(f,\Lambda).$ Moreover, for any $\overline{\nu}\in M(f,\Lambda)$, by Lemma \ref{lem-decomposition}, there is a $\nu\in M_{erg}(f^k,\Lambda)$, an $m|k$ and a $X_0\subseteq X$ satisfying the properties there. Since $f^{k}\Delta=\Delta$, $\nu(f^i\Delta)=1$ for some $0\leq i\leq k-1.$ Therefore, $f_*^{k-i}\nu\in M_{erg}(f^k,\Delta)$ and $j(f_*^{k-i}\nu)=\overline{\nu}$, proving that $j:(M_{erg}(f^k,\Delta),\rho_1) \to (M_{erg}(f,\Lambda),\rho_2)$ is surjective.


      (3) Note that $h_{f_*^i\nu}(f^k|_{f^i\Delta})=h_{f_*^i\nu}(f^k|_\Lambda),0\leq i\leq k-1$. Moreover,
      $$h_\nu(f^k|_\Delta)=h_{f_*\nu}(f^k|_{f\Delta})=\cdots=h_{f_*^{k-1}\nu}(f^k|_{f^{k-1}\Delta}).$$
      So
      $$kh_{\overline{\nu}}(f)=h_{\overline{\nu}}(f^k)=\frac1k\sum_{i=0}^{k-1}h_{f_*^i\nu}(f^k|_\Lambda)=\frac1k\sum_{i=0}^{k-1}h_{f_*^i\nu}(f^k|_{f^i\Delta})=h_\nu(f^k|_\Delta).$$

\end{proof}

We also need the following result whose proof shall be given in the appendix.
\begin{mainproposition}\label{prop-symbolic}
Consider a m.  f.  t.  -subshift $M$ and a continuous function $\varphi$ on $M$.   Let $\{\mu^i\}_{i=1}^k,  k\geq 1$ be $k$ distinct ergodic measure on $M$.
Then for any $\eta>0$ and any $\eps>0,$
there exists a minimal subsystem $\overline{M}\subseteq M$ supporting exactly $k$ ergodic measures $\{\nu^j\}_{j=1}^k$ such that\\
 (1) \begin{equation}\label{integral}
  \rho(\nu^j,  \mu^j)<\eps/2~\textrm{for each}~j=1,  \cdots,  k.
 \end{equation}
 (2) $h_{\nu^j}>h_{\mu^j}-\eta,  j=1,  \cdots,  k$.
\end{mainproposition}


For any $m\in\N$ and $\{\nu_i\}_{i=1}^m \subseteq M(X)$,   we write $\cov\{\nu_i\}_{i=1}^m$ for the convex combination of $\{\nu_i\}_{i=1}^m$,   namely,
$$\cov\{\nu_i\}_{i=1}^m:=\left\{\sum_{i=1}^mt_i\nu_i:t_i\in[0,  1],  1\leq i\leq m~\textrm{and}~\sum_{i=1}^mt_i=1\right\}.  $$

\begin{Def}
We say that $(X,  f)$ has the minimal-entropy-dense property if for any $K=\cov\{\mu_i\}_{i=1}^m\subseteq M(f,  X)$ and any $\eta,  \zeta>0$,   there exists a compact invariant subset $\overline{X} \subsetneq X$,    such that
\begin{enumerate}
 \item there are exactly $k$ ergodic measures $\{\nu_i\}_{i=1}^m$ on $\overline{X}$;
  \item For each $1\leq i\leq m$,
  $h_{\nu_i}>h_{\mu_i}-\eta$ and consequently,   $\htop(f,  \overline{X})>\sup\{h_{\kappa}:\kappa\in K\}-\eta$.
  \item  For each $1\leq i\leq m$, $d(\nu_i,\mu_i)<\zeta$ and consequently,  $d_H(K,  M(f,  \Lambda))<\zeta$. 
\end{enumerate}


\end{Def}

{
\begin{mainlemma}\label{lem-minimalentropy-dense}
Suppose $(X,  f)$ is topologically expanding (resp.   topologically hyperbolic).
   Then $(X,  f)$ has the minimal-entropy-dense property.
\end{mainlemma}

{
\begin{proof}

Consider $K:=\cov\{\mu_i\}_{i=1}^k.  $  Fix $\eta>0,\zeta>0.$ We may assume that $\zeta<\frac12 \min\{d(\mu_i,\mu_j)|\,1\leq i\neq j\leq k\}.$  By Lemma \ref{Mainlemma-convex-by-horseshoe},      there exist  closed $f$-invariant   $\Lambda,  \Lambda_i$  satisfying that
\begin{enumerate}
  \item $\Lambda$ is transitive and has the shadowing property.
  \item For each $1\leq i\leq m$,
  $\htop(f,  \Lambda_i)>h_{\mu_i}-\eta$ and consequently,   $\htop(f,  \Lambda)>\sup\{h_{\kappa}:\kappa\in K\}-\eta$.
  \item $d_H(K,  M(f,  \Lambda))<\zeta$,   $d_H(\mu_i,  M(f,  \Lambda_i))<\zeta$.
\end{enumerate}
By the variational principle,   we choose $k$ ergodic measures $\omega_i\in M(f,  \Lambda_i),  i=1,  \cdots,  k$ such that
\begin{equation}\label{eq-nu-omega}
  h_{\omega_i}>\htop(f,  \Lambda_i)-\eta>h_{\mu_i}-2\eta~\textrm{and}~\rho_X(\omega_i,  \mu_i)<\zeta/2.
\end{equation}
From the proof of Lemma \ref{Mainlemma-convex-by-horseshoe} in \cite{DongTian2016-nosyndetic} in fact there exist     closed $f^k$-invariant sets $\Delta,\Delta_i$ such that $\pi:(\Delta,  f^n)\to (\Sigma_r^+,  \sigma)$ and $\pi_i:(\Delta_i,  f^n)\to(\Sigma_r^+,  \sigma)$ are conjugations,    $\Lambda=\bigcup_{l=0}^{n-1}f^l\Delta$ and $\Lambda_i=\bigcup_{l=0}^{n-1}f^l\Delta_i$.

By Proposition \ref{prop-f-f-k},   there is an isometric isomorphism $j$ between $M_{erg}(f^n,  \Delta)$ and $M_{erg}(M,  \Lambda)$.   So for each $\omega_i$,   we obtain an $\widetilde{\omega}_i=j^{-1}(\omega_i)\in M_{erg}(f^n,  \Delta).  $ Moreover,   each $\widetilde{\omega}_i$ is pushed forward by $\pi$ to an ergodic measure $\omega^i=\pi_*\widetilde{\omega}_i\in\M_\sigma(\Sigma_r^+)$.   Then we use Proposition \ref{prop-symbolic} to find a minimal subsystem $\overline{M}\subseteq \Sigma_r^+$ supporting exactly $k$ ergodic measures $\{\nu^j\}_{j=1}^k$ such that
\begin{description}
 \item[1] $\rho(\nu^j,  \omega^j)<\zeta/2~\textrm{for each}~j=1,  \cdots,  k.  $

  \item[2] $h_{\nu^j}>h_{\omega^j}-\eta,  j=1,  \cdots,  k$.
  \end{description}
Let $\overline{X}=\pi^{-1}(\overline{M})$.   It is left to show that $\overline{X}$ is the one we need.   Indeed,   since $\pi$ is a conjugation,   $(\overline{X},  f^n)$ supports exactly $k$ ergodic measures $\widetilde{\nu}_i:=(\pi^{-1})_*\nu^i\subseteq M(f^n,  \overline{X})$.   Note that $\overline{X}$ is $f^n$-minimal and $AP(f^n)=AP(f)$.   So $\overline{X}$ is $f$-minimal.   Now let $\nu_i=j(\widetilde{\nu}_i)$.   Then
$$\rho_{\Delta}(\widetilde{\omega}_i,  \widetilde{\nu}_i)=\rho_{\Sigma_r^+}(\omega^i,  \nu^i)<\zeta/2.  $$
Moreover,   from \cite{Walters} we see
\begin{equation}\label{eq-h-nu-omega}
  h_{\widetilde{\nu}_i}(f^n)=h_{\nu^i}(\sigma)>h_{\omega^i}(\sigma)-\eta=h_{\widetilde{\omega}_i}(f^n)-\eta.
\end{equation}
Furthermore,   by Proposition \ref{prop-f-f-k},   we see
$$\rho_X(\omega_i,  \nu_i)=\rho_{\Delta_i}(\widetilde{\omega}_i,  \widetilde{\nu}_i)<\zeta/2.  $$
This along with \eqref{eq-nu-omega} shows that $\rho_X(\mu_i,  \nu_i)<\zeta$ for each $i=1,  \cdots,  k.  $
In addition,   Proposition \ref{prop-f-f-k} gives that
$$nh_{\nu_i}(f)=h_{\widetilde{\nu}_i}(f^n)~\textrm{and}~h_{\widetilde{\omega}_i}(f^n)=nh_{\omega_i}(f).  $$
This along with    \eqref{eq-nu-omega} and \eqref{eq-h-nu-omega} shows that
$$h_{\nu_i}>h_{\mu_i}-3\eta.  $$
The proof is completed.
\end{proof}
}


}

\section{Proof of Main Theorems}\label{Proof}

We consider a topological dynamical system $(X,  f)$ and suppose  that $\Lambda\subsetneq X$ is a closed $f$-invariant set and  $\varphi$ is a continuous function on $X$.
Define $I^\Lambda_{\varphi}(f)=\{x\in I_\varphi(f)|\,   V_f(x)\subseteq M(f,  \Lambda)\}$.
Denote
$$L^\Lambda_\varphi=\left[\inf_{\mu\in M(f,  \Lambda)}\int\varphi d\mu,  \,  \sup_{\mu\in M(f,  \Lambda)}\int\varphi d\mu\right]=[L_1^\Lambda,  L_2^\Lambda].  $$
For $a\in L^\Lambda_\varphi$,   define $R^\Lambda_{\varphi}(a)=\{x\in R_\varphi(a)|\,   V_f(x)\subseteq M(f,  \Lambda)\}$ and denote
\begin{equation}\label{t-a-Lambda}
  t_a^\Lambda=\sup_{\mu\in M(f,  \Lambda)}\{h_\mu:\,  \int\varphi d\mu=a\}.
\end{equation}
  Define the measure center of an invariant set $\Lambda\subseteq X$ as
$$C^*_\Lambda:=\overline{\bigcup_{\mu\in M(f,  \Lambda)}S_{\mu}}.  $$ In particular, for any $x\in X$,   we define the measure center of $x$ as
 $C^*_x:=C^*_{\omega_f(x)}.  $

\subsection{Proof of Theorem \ref{thm-CCCCCCCCCCCCC} and Corollary \ref{Corollary-thm-CCCCCCCCCCCCC}}

\begin{mainlemma}\label{lem-irregular-horseshoe}
Suppose $(X,  f)$ is topologically expanding and transitive.   Let $\varphi:X\rightarrow \mathbb{R}$ be a continuous function and assume that $Int(L_\varphi)\neq \emptyset.  $
 Then for any $\eta>0, k\geq 2, $ there exist two  $f$-invariant subsets
  $\Lambda\subsetneq \Theta \subsetneq X$ such that
 \begin{description}
   \item[(1)] $\Lambda$ is minimal  with exactly $k$ ergodic measures and satisfies
   \begin{equation}\label{eq-Lambda-irregular}
  C^*_\Lambda=\Lambda,  ~\Int(L^\Lambda_\varphi)\neq \emptyset,  ~h_{top}(f|_\Lambda)>h_{top}(f)-\eta.
\end{equation}
   \item[(2)] $\Theta$ is internally chain transitive but not topologically transitive and
   \begin{equation}\label{eq-Delta-irregular}
  \Lambda=C^*_{\Theta}\subsetneq\Theta,  ~\Int(L^\Theta_\varphi)\neq \emptyset,  ~h_{top}(f|_\Theta)>h_{top}(f)-\eta.
\end{equation}

 \end{description}

\end{mainlemma}
\begin{proof}
(1) By the variational principle,   we obtain a $\mu\in M_{erg}(f,  X)$ with $h_{\mu}>\htop(f)-\eta/2$.   Since $Int(L_\varphi)\neq \emptyset$,   we obtain a $\nu\in M_{erg}(f,  X)$ with $\int\varphi d\nu\neq\int\varphi d\mu$.   Then we choose $\theta\in(0,  1)$ close to $1$ such that $\omega:=\theta\mu+(1-\theta)\nu$ satisfies that
$$h_{\omega}=\theta h_{\mu}+(1-\theta)h_{\nu}\geq \theta h_{\mu}>\htop(f)-\eta/2.  $$
Moreover,   one note that $|\int\varphi d\omega-\int\varphi d\mu|=(1-\theta)|\int\varphi d\mu-\int\varphi d\nu|\neq0$.   Now let
$$K:=\{t\mu+(1-t)\omega:0\leq t\leq1\}.  $$
By Lemma \ref{lem-minimalentropy-dense},   we obtain a minimal
 set $\Lambda\subsetneq X$  with exactly $k$ ergodic measures and there are two ergodic measures $\widetilde{\mu},  \widetilde{\omega}\in M(f,  \Lambda)$ such that
 $$h_{\widetilde{\mu}}>h_{\mu}-\eta/2>\htop(f)-\eta,  ~h_{\widetilde{\omega}}>h_{\omega}-\eta/2>\htop(f)-\eta$$
 and
 $$\left|\int\varphi d\widetilde{\mu}-\int\varphi d\mu\right|<\eps/3,  ~\left|\int\varphi d\widetilde{\omega}-\int\varphi d\omega\right|<\eps/3,  $$
 where $\eps:=|\int\varphi d\mu-\int\varphi d\omega|>0$.   By the triangle inequality,   one sees that $|\int\varphi d\widetilde{\mu}-\int\varphi d\widetilde{\omega}|>\eps-\eps/3-\eps/3=\eps/3$.   This implies that $Int(L^\Lambda_\varphi)\neq \emptyset.  $ Moreover,
 $$h_{top}(f|_\Lambda)\geq h_{\widetilde{\mu}}>h_{top}(f)-\eta.  $$
  It is clear that $C^*_\Lambda=\Lambda $, since $\Lambda$ is minimal.

 (2)  
  Since $(X,  f)$ is topologically expanding and transitive,   we can find a $z\notin\Lambda$ such that $\omega_f(z)=\Lambda$ by Lemma \ref{Lem-locally-omega-f-equalto-ICT}.   Let $A=\orb(z,  f)\cup \Lambda$.   Note that $A$ is closed and $f$-invariant.   So by Lemma \ref{lem-omega-A},   there is a point $x\in X$ such that
 $$A\subseteq\omega_f(x)\subseteq \cup_{l=0}^{\infty}f^{-l}A.  $$
 In particular,
$$\overline{\bigcup_{y\in\omega_f(x)}\omega_f(y)}\subseteq A\subseteq\omega_f(x)\neq X.  $$
Let $\Theta=\omega_f(x)$.   Then $\Theta$ is internally chain transitive  by Lemma \ref{Lem-omega-naturally-in-ICT}.   By Lemma \ref{lem-Lambda-Lambda-0},
$$M(f,  \Theta)=M(f,  A)=M(f,  \Lambda).  $$
This implies that $C^*_\Theta=\Lambda\subsetneq\Theta$ and $\Int(L_\varphi^\Theta)=\Int(L_\varphi^\Lambda)\neq\emptyset$.   It is left to show that $\Theta$ is not transitive.   Indeed,   suppose the opposite is true,   then there is an $a\in\omega_f(x)$ such that $\omega_f(a)=\omega_f(x)$.   However,   $a=f^nb$ for some $n\in\N$ and $b\in A.  $ In particular,   $\omega_f(a)\subseteq \Lambda\subsetneq \omega_f(x)$,   a contradiction.
\end{proof}

{\bf Proof of Theorem \ref{thm-CCCCCCCCCCCCC}}

 For any $\eta>0, k\geq 2$, we use Lemma \ref{lem-irregular-horseshoe}  to find a minimal subset $\Lambda$ such that $\Lambda$ supports exactly $k$ ergodic measures and there are two ergodic measures $\nu_1,  \nu_2\in M(f,\Lambda)$ with $h_{\nu^1}>h_\nu-\eta/2>\htop(f)-\eta$ and
      $$\int\varphi d\nu^1<\int\varphi d\nu^2.  $$
      Now choose $0<\theta_1<\theta_2<1$ close to $1$ such that $\mu^i=\theta_i\nu^1+(1-\theta_i)\nu^2,  i=1,  2$ satisfy that
      $$h_{\mu^i}\geq \theta_i h_{\nu^1}>h_{\nu}-\eta/2>\htop(f)-\eta.  $$
      Let $K:=\cov\{\mu^1,  \mu^2\}$.   Then by Theorem \ref{thm-density-basic-property} it is not hard to see that 
      $$G_K^{\Lambda,  N}\subseteq \{x\in X\,  |\,  x\text{ satisfies Case (1) } \}\cap \mathbb{E}_k\cap NRec(f)\cap I_{\varphi}(f).  $$
      Moreover,   Lemma \ref{prop-locally-star-implies-nonrecurrently-star} indicates that $(\Lambda,  f)$ is nonrecurrently-star-saturated.   So
      $$\htop(G_K^{\Lambda,  N})=\inf\{h_{\mu^1},  h_{\mu^2}\}>\htop(f)-\eta.  $$
          Since $\eta$ is arbitrary,   we see that
          $\htop(\{x\in X\,  |\,  x\text{ satisfies Case (1) } \}\cap \mathbb{E}_k\cap NRec(f)\cap I_{\varphi}(f))=\htop(f)>0,\, i=1, 2,\,k=2,3,4,\cdots.  $

Let $\Theta$ be the set of Lemma \ref{lem-irregular-horseshoe}.
One can follow above way by changing $\Lambda$ to $\Theta$ to give the proof that   $\htop(\{x\in X\,  |\,  x\text{ satisfies Case (2) } \}\cap \mathbb{E}_k\cap NRec(f)\cap I_{\varphi}(f))=\htop(f)>0,\, i=1, 2,\,k=2,3,4,\cdots.  $ \qed



\bigskip

{\bf Proof of Corollary \ref{Corollary-thm-CCCCCCCCCCCCC}}

Since the considered dynamical system is not uniquely ergodic, there are two invariant measures $\mu_1\neq \mu_2$ and thus
there is a   continuous function
$\varphi:X\rightarrow \mathbb{R}$
such that
 $ \int\varphi d\mu_1\neq  \int\varphi d\mu_2.$ So $Int(L_\varphi)\neq \emptyset$.
 Then Theorem \ref{thm-CCCCCCCCCCCCC} implies Corollary \ref{Corollary-thm-CCCCCCCCCCCCC}, since  $I_\varphi(f)\subseteq I(f). $ \qed

\subsection{Proof of Theorem \ref{thm-DDDDDDDDDD}, Theorem \ref{thm-A0000000000} and  Corollary \ref{Cor--thm-DDDDDDDDDD} }

\begin{mainlemma}\label{lem-regualr-horseshoe}
Suppose $f: X\rightarrow X$ is transitive and topologically expanding.   Let $\varphi:X\rightarrow \mathbb{R}$ be a continuous function and assume that $Int(L_\varphi)\neq \emptyset.  $
 Then for any $a\in Int(L_\varphi)$ and any $\eta>0,  k\geq 2, $  there exist two  $f$-invariant subsets
  $\Lambda\subsetneq \Theta \subsetneq X$ such that  $\Lambda$ is minimal  with exactly $k$ ergodic measures $\{\omega_i\}_{i=1}^k$ satisfying that
\begin{equation}\label{eq-sup-leq-inf}
   \int\varphi(x)d\omega_1<a< \int\varphi(x)d\omega_2<\int\varphi(x)d\omega_3<\cdots\int\varphi(x)d\omega_k,
\end{equation}
   and
\begin{equation}\label{eq-Lambda-regular}
\min_{1\leq i\leq k}\{h_{\omega_i}\}>t_a-\eta.
\end{equation}
Moreover,   
 $\Theta$   is internally chain transitive but not topologically transitive such that
\begin{equation}\label{eq-Delta-regular}
  \Lambda=C^*_{\Theta}\subsetneq \Theta~\textrm{and}~t_a^{\Theta}>t_a-\eta.
\end{equation}
\end{mainlemma}

{
\begin{proof}
Since $M(f,  X)$ is compact and $\mu\mapsto\int\varphi d\mu$ is continuous,   there exist $\mu_{\max},  \mu_{\min}\in M(f,  X)$ with
$$\int\varphi d\mu_{\max}=\sup_{\mu\in M(f,  X)}\int\varphi d\mu~\textrm{and}~\int\varphi d\mu_{\min}=\inf_{\mu\in M(f,  X)}\int\varphi d\mu.  $$
For any $\eta>0$,   by the definition of $t_a$,   there exists a $\lambda\in M(f,  X)$ with $\int\varphi d\lambda=a$ such that $h_{\lambda}>t_a-\eta/2$.   Now choose $\theta\in(0,  1)$ close to $1$ such that $\nu_1:=\theta\lambda+(1-\theta)\mu_{\min}$ and $\nu_k=\theta\lambda+(1-\theta)\mu_{\max}$ satisfy that
$$h_{\nu_1}=\theta h_{\lambda}+(1-\theta)h_{\mu_{\min}}\geq\theta h_{\lambda}>t_a-\eta/2~\textrm{and}
~h_{\nu_k}=\theta h_{\lambda}+(1-\theta)h_{\mu_{\max}}\geq\theta h_{\lambda}>t_a-\eta/2.  $$
 Take $\theta_i\in(\theta,1)$ with the order that $\theta_2>\theta_3>\cdots>\theta_{k-1}$.
 Then we let $\eps:=\min\{a-\int\varphi d\nu_1,  \int\varphi d\nu_2-a,  \int\varphi d\nu_3-\int\varphi d\nu_2, \cdots, \int\varphi d\nu_k-\int\varphi d\nu_{k-1}\}>0$ and choose $\zeta>0$ such that
$$\rho(\tau,  \kappa)<\zeta\Rightarrow |\int\varphi d\tau-\int\varphi d\kappa|<\eps/2.  $$
By Lemma \ref{lem-minimalentropy-dense},   there are a minimal
 set $\Lambda\subsetneq X$  with exactly $k$ ergodic measures $\omega_i$ such that $d(\omega_i,\nu_i)<\zeta$, $h_{\omega_i}>h_{\nu_i}-\eta/2,  i=1,  2,\cdots, k$.     This implies that
 $|\int\varphi d\omega_i-\int\varphi d\nu_i|<\eps/2$.   So we have
 $$\int\varphi(x)d\omega_1<a< \int\varphi(x)d\omega_2<\int\varphi(x)d\omega_3<\cdots\int\varphi(x)d\omega_k.$$



The construction method of $\Theta$ is same as Lemma \ref{lem-irregular-horseshoe}, here we omit the details.
\end{proof}

}




 {\bf Proof of Theorem \ref{thm-DDDDDDDDDD}}

The part $\leq $ can be deduced from Lemma \ref{lem-Bowen} directly (letting $t=t_a$).

By Lemma \ref{lem-regualr-horseshoe}, for any $a\in Int(L_\varphi)$ and any $\eta>0,  k\geq 2, $  there exist two  $f$-invariant subsets
  $\Lambda\subsetneq \Theta \subsetneq X$ such that  $\Lambda$ is minimal  with exactly $k$ ergodic measures $\{\omega_i\}_{i=1}^k$ satisfying \eqref{eq-sup-leq-inf}, \eqref{eq-Lambda-regular} and \eqref{eq-Delta-regular}.

Take suitable $\theta_1\in(0,1)$ such that $\mu_{1}:=\theta_1\omega_1+(1-\theta_1)\omega_{2}$ satisfies that $\int \varphi d\mu_1=a.$
If $k\geq 3,$ take suitable $\theta_2\in(0,1)$ such that $\mu_{2}:=\theta_2\omega_1+(1-\theta_2)\omega_{3}$ satisfies that $\int \varphi d\mu_2=a.$ Remark that $\mu_1\neq \mu_2$ since $\omega_i$ are different ergodic measures.

      Let $K_1:=\{\mu_1\}$ and $K_2:=\cov\{\mu_1,  \mu_2\}$.   Then by Theorem \ref{thm-density-basic-property} it is not hard to see that
      $$G_{K_1}^{\Lambda,  N}\subseteq \{x\in X\,  |\,  x\text{ satisfies Case (1) } \}\cap \mathbb{E}_k\cap NRec(f)\cap R_{\varphi}(a)\cap QR  $$
      and  $$G_{K_2}^{\Lambda,  N}\subseteq \{x\in X\,  |\,  x\text{ satisfies Case (1) } \}\cap \mathbb{E}_k\cap NRec(f)\cap R_{\varphi}(a)\cap IR.  $$
      Moreover,   Lemma \ref{prop-locally-star-implies-nonrecurrently-star} indicates that $(\Lambda,  f)$ is nonrecurrently-star-saturated.   So combined with  \eqref{eq-Lambda-regular} we have
      $$\htop(G_{K_1}^{\Lambda,  N})= h_{\mu_1} >t_a-\eta,\,\, \htop(G_{K_2}^{\Lambda,  N})=\inf\{h_{\mu_1},  h_{\mu_2}\}>t_a-\eta.  $$
          Since $\eta$ is arbitrary,   we get the proof of Case (1) in  item (I) and (II).

          One can follow above way by changing $\Lambda$ to $\Theta$ and using \eqref{eq-Delta-regular} to give the proof  of Case (2) in  item (I) and (II). \qed



\bigskip


{\bf Proof of Theorem \ref{thm-A0000000000}}


(1) Since the considered dynamical system is not uniquely ergodic, there are two invariant measures $\mu_1\neq \mu_2$ and thus
there is a   continuous function
$\varphi:X\rightarrow \mathbb{R}$
such that
 $ \int\varphi d\mu_1\neq  \int\varphi d\mu_2.$ So $Int(L_\varphi)\neq \emptyset$.
 Fix $\varepsilon>0.$ One can  take a number $a\in Int (L_\varphi)$ such that $$h_{top} (R_{\varphi} (a))> h_{top} (f)-\varepsilon.$$
 Recall that   $R_\varphi (f)=\bigsqcup_{b\in\mathbb{R}}R_{\varphi} (b).$  So item (I) except $k=1$   can be deduced from item (I) of  Theorem \ref{thm-DDDDDDDDDD}.

If $k= 1$, for $\varepsilon>0,$ one can take a minimal set $\Lambda$ with exactly one  ergodic measure $\mu$ such that $h_{top}(\Lambda)=h_\mu>h_{top} (f)-\varepsilon.$ Let $K=\{\mu\}$. Note that $G_K^{\Lambda,N}\subseteq \{x\in X\,  |\,  x\text{ satisfies Case (1) } \}\cap \mathbb{E}_1\cap NRec(f)\cap   QR .$ One can use  the fact that  $(\Lambda,  f)$ is nonrecurrently-star-saturated by Lemma \ref{prop-locally-star-implies-nonrecurrently-star} to prove $G_K^{\Lambda,N}$ has entropy larger than $h_{top} (f)-\varepsilon$. This completes the proof for Case (1). For Case (2), one can follow the proof of Lemma  \ref{lem-irregular-horseshoe} to construct a set $\Theta$ such that $C^*_\Theta=\Lambda\subsetneq \Theta$ and then one can follow above way by changing $\Lambda$ to $\Theta$ to give the proof.

(2) For $k\geq 1$ and $\varepsilon>0,$ one can take a minimal set $\Lambda_k$ with exactly $k$ ergodic measures  such that $h_{top}(\Lambda_k)>h_{top} (f)-\varepsilon.$ By variational principle, there is an ergodic measure $\mu$ such that $h_\mu=h_{top}(\Lambda_k)>h_{top} (f)-\varepsilon.$ Note that $\mu(G_\mu\cap \Lambda_k)=1$ so that $$h_{top}(G_\mu\cap \Lambda_k)\geq h_\mu >h_{top} (f)-\varepsilon.$$ Note that $G_\mu\cap \Lambda_k\subseteq \{x\in X\,  |\,  x\text{ satisfies Case (1) } \}\cap \mathbb{E}_k\cap Rec(f)\cap   QR $ so that item (II) is proved.\qed

\bigskip


{\bf Proof of Corollary \ref{Cor--thm-DDDDDDDDDD}}

Note that $QR\subseteq R_\varphi$ so that item (I) and (III) can be deduced from Theorem   \ref{thm-A0000000000}. Now we consider item (II).

 Firstly we consider the case of  $I_\varphi (f)\neq\emptyset.$  It implies that  $Int(L_\varphi)\neq \emptyset$. Fix $\varepsilon>0.$ One can  take a number $a\in Int (L_\varphi)$ such that $$h_{top} (R_{\varphi} (a))> h_{top} (f)-\varepsilon.$$
 Recall that   $R_\varphi (f)=\bigsqcup_{b\in\mathbb{R}}R_{\varphi} (b).$  So   item (II) of Corollary \ref{Cor--thm-DDDDDDDDDD} can be deduced from  item (II) of Theorem \ref{thm-DDDDDDDDDD}.

 If $I_\phi (f)=\emptyset$, then $R_\phi(f)=X$ so that (II) of Corollary \ref{Cor--thm-DDDDDDDDDD} is same as Corollary \ref{Corollary-thm-CCCCCCCCCCCCC} which has been proved.  \qed


\subsection{Non-expansive case: Proof of Theorem \ref{thm-AABBBB}}


\begin{Lem} \cite{DOT}\label{horseshoe}
Suppose that $(X,  f)$ has the shadowing property and $\htop(f)>0$.    Then for any $0<\alpha<\htop(f)$ there are $m,  k\in \N$,   $\log (m)/k>\alpha$ and a closed set $\Lambda\subseteq X$
   invariant under $f^k$ such that there is a factor map $\pi\colon (\Lambda,   f^k)\to (\Sigma_{m+1}^+,  \sigma)$.
\end{Lem}


\begin{Lem}\cite{DongTian2016-nosyndetic} \label{horseshoe2}
Suppose $(X,  f)$ satisfies   
 almost specification
  property.   Then for any $0<\alpha<\htop(f)$,   there are $m,  k\in \N$,   $\log (m)/k>\alpha$ and a closed set $\Lambda\subseteq X$
   invariant under $f^k$ such that there is a factor map $\pi\colon (\Lambda,   f^k)\to (\Sigma_{m}^+,  \sigma)$.
\end{Lem}


{\bf Proof of Theorem \ref{thm-AABBBB}.}
When $\htop(f)=0$,   there is nothing to prove.   So we suppose $\htop(f)>0$.   By Lemma \ref{horseshoe} and \ref{horseshoe2},   for any $0<\alpha<\htop(f)$,   there are $m,  k\in \N$,   $\log (m)/k>\alpha$ and a closed and $f^k$-invariant set $\Lambda\subseteq X$ with a semiconjugation $\pi:(\Lambda,  f^k)\to(\Sigma_m^+,  \sigma)$.



 Fix $l\geq 1.$ Since $(\Sigma_m^+,  \sigma)$ is topologically expanding and transitive,   by item (II) of Theorem \ref{thm-A0000000000},   we obtain that $\htop(AP(\sigma)\cap   \mathbb{E}_l(\sigma))=\htop(\sigma,  \Sigma_m^+)>k\alpha$.  Now $\pi^{-1}(\Sigma)$ is a closed and $f^k$-invariant set   of $X$.
      By Zorn Lemma,   there exists a $f^k$-minimal set $\Delta\subseteq \pi^{-1}(\Sigma)$.   Then $\pi(\Delta)$ is $\sigma|_{\Sigma}$-invariant and minimal.   Since $\pi(\Delta)\subseteq \Sigma$ and $\Sigma$ is minimal,   $\pi(\Delta)=\Sigma.  $ Then $\pi|_\Delta:(\Delta,  f^k)\to(\Sigma_m^+,  \sigma)$ is a semi-conjugation so that
      $\htop(f^k,  \Delta)\geq \htop(\sigma,  \Sigma)\geq k\alpha$.   Whereby $\htop(f,  \Delta)=\frac1k\htop(f^k,  \Delta)\geq\alpha$.   Meanwhile,   $AP(f^k)=AP(f)$. So $\Delta\subseteq AP(f)$.  By Proposition \ref{prop-f-f-k} $\Delta\subseteq \cup_{t\geq l}\mathbb{E}_t(f^k)\subseteq \cup_{t\geq l}\mathbb{E}_t(f)$.  Since  and  $\alpha$   is arbitrary,   the proof is completed. \qed



{
\subsection{  Proof of Theorem \ref{theorem-minimal-irregular}.} Fix $\,k\geq2.  $  By minimal-entropy-dense propety, there exists a minimal set $\Lambda$ with exactly $k$ ergodic measures such that $\inf_{\mu\in M(f|_\Lambda,\Lambda)}\varphi d\mu < \sup_{\mu\in M(f|_\Lambda,\Lambda)}\varphi d\mu.$ Now we only need to prove that there exists $x_0\in \Lambda \cap I_{\varphi}(f).$

Otherwise, for any $x\in \Lambda$, the limit $\lim_{n\rightarrow \infty}\frac1n\sum_{i=0}^{n-1}\varphi(f^ix)$ exists. Denote the limit function by $\varphi^*$.   $\inf_{\mu\in M(f|_\Lambda,\Lambda)}\int \varphi d\mu < \sup_{\mu\in M(f|_\Lambda,\Lambda)}\int \varphi d\mu$ implies that there exist  two ergodic measures $\mu_1,\mu_2$ supported on $\Lambda$ with $\int \varphi d\mu_1\neq \int \varphi d\mu_2$ so that there exist points $x_1,x_2\in\Lambda$ $\varphi^*(x_1)\neq \varphi^*(x_2).$ Note that $\varphi^*:\Lambda\rightarrow \Lambda$ is the limit of convergent sequence of continuous functions on compact set $\Lambda$. So by Baire theorem, $\varphi^*:\Lambda\rightarrow \Lambda$ has a point $y\in\Lambda$ such that $\varphi^*$ is continuous at $y.$ Let $a:=\varphi^*(y).$ Since $\Lambda$ is minimal, then for  any $x\in \Lambda$, there exists $n_j\rightarrow +\infty$ such that $d(f^{n_j}(x),y)\rightarrow 0$ as $j\rightarrow +\infty$. Thus $\varphi^*(f^{n_j}(x))\rightarrow\varphi^*(y)=a$    as $j\rightarrow +\infty$. Since $\varphi^*$ is constant on every orbit of $\Lambda,$ $\varphi^*(x)=\varphi^*(f^{n_j}(x))=\lim_{j\rightarrow +\infty}\varphi^*(f^{n_j}(x))=\varphi^*(y)=a.$ Thus $\varphi^*:\Lambda\rightarrow \Lambda$ is a constant function. It contradicts that $\varphi^*(x_1)\neq \varphi^*(x_2).$ \qed

}

\section{Appendix: Proof of Proposition \ref{prop-symbolic}}
\subsection{Notions and Notations}
Fix any positive integer $k\geq 2$ and consider the following set
$\Sigma_k^+=\{0,  1,  \dotsc,  k-1\}^{\N_0}$ with the product topology induced by the discrete topology on $\{0,  1,  \dotsc,  k-1\}$.
The space $\Sigma_k^+$ is always endowed with the \textit{shift map} $\sigma$ defined by $\sigma(x)_i=x_{i+1}$ for every integer $i\geq 0$.
It is not hard to verify that $\sigma$ is continuous and $\Sigma_k^+$ is a compact metrizable space.   We endow it with the (compatible) metric
defined by
$d(x,  y)=2^{-k}$ when $x\neq y$ and $k=\min\{i : x_i\neq y_i\}$.   Dynamical system $(\Sigma_k^+,  \sigma)$ or simply $\Sigma_k^+$ for short,
is called \textit{full shift}.   By \textit{subshift} or \textit{shift} we mean any compact and $\sigma$-invariant subset of $\Sigma_k^+$.

A finite sequence $(a_1,  \cdots,  a_N)$ of elements $a_j\in S$ is called a block (or N-block) in $\Sigma_k^+$.   $N$ is called the length of the block.   The block $A=(a_1,  \cdots,  a_N)$ is said to occur in $x\in \Sigma_k^+$ at the place of $m$ if $x_m=a_1,  \cdots,  x_{m+N-1}=a_N$.   In this case one writes $A\prec x$.   The block $A$ is said to occur in the subshift $\Lambda\subseteq \Sigma_k^+$ if there exists an $x\in \Lambda$ such that $A\prec x$.   In this case,   one writes $A\prec\Lambda$.

For any $m\in Z$ and any block $A=(a_1,  \cdots,  z_N)$ in $\Sigma_k^+$.   Let ${}_m[a_1,  \cdots,  a_N]$ denote the set of all $x\in \Sigma_k^+$ such that $(a_1,  \cdots,  a_N)$ occurs in $x$ at the place $m$.   The set ${}_m[a_1,  \cdots,  a_N]$ is called a cylinder of length $N$ based on the block $(a_1,  \cdots,  a_N)$ at the place $m$.

It follows easily from the product topology that a cylinder is both open and closed.   Then it is not hard to show that $\Span\{1_P:P\in Bl(X)\}$ is dense in $C(X)$ via Stone-Weierstrass Theorem.   Here $1_P$ denotes the characteristic function of $P$ and $\Span A$ denotes the linear space generated by $A$.
So we can define a metric on $\Sigma_k^+$ for the $weak^{*}$ topology as follows.
$$\rho(\mu,  \nu)=\sum_{r\geq1 \atop P\in Bl_r(\Omega_k)}\frac{|\mu(P)-\nu(P)|}{2^r\cdot k^r}.  $$
Therefore,   we have the following characterization.
\begin{Lem}\label{measure-on-symbolic}
Then for any $\eps>0$,   there exist $t=t(\eps)>0$ and $\delta=\delta(t,  \eps)>0$ such that
$$\mu,  \nu\in M(\Sigma_k^+),  ~|\mu(P)-\nu(P)|<\delta~\forall~P\in Bl_t(\Sigma_k^+)\Rightarrow \rho(\mu,  \mu)<\eps.  $$
\end{Lem}

We denote the set of all blocks by $Bl$ and the set of blocks of length $r$ by $Bl_r$.   For $P=(p_1,  \cdots,  p_m)\in Bl_m,   Q\in(q_1,  \cdots,  q_n)\in Bl_n$,   we write $PQ=P\cdot Q=(p_1,  \cdot,  p_m,  q_1,  \cdot,  q_n)\in Bl_{m+n}$,   the juxtaposition of $P$ and $Q$.
If $l(Q)\geq l(P)$,   we define
$$\mu_Q(P)=(l(Q)-l(P)+1)^{-1}\cdot\card\{j:1\leq j\leq l(Q)-l(P)+1,  (q_j,  \cdot,  q_{j+l(P)-1})=P\}$$
as the relative frequency of (the occurrence of) $P$ in $Q$ as a subblock.   Then $\mu_Q$ is a probability measure on $Bl_m$ for every $m\leq l(Q)$.   If $Q$ occurs in $x$ at the place $m$,   then one sees that
\begin{equation}\label{eq-P-Q-x}
  \mu_Q(P)=\langle 1_P,  \E_{l(Q)-l(P)+1}(\sigma^m(x))\rangle.
\end{equation}

For $w\in \Sigma_k^+$,   $O_w=\{\sigma^jw~|~j\in\Z\}$ is the orbit of $w$ and if $s\leq t\in \Z$:
$$w_{\langle s,  t\rangle}=(w_s,  \cdots,  w_t)\in Bl_{t-s+1}.  $$
For a subshift $K$ of $\Sigma_k^+$ and $r\in \N$,
$$Bl_r(K)=\{P\in Bl_r~|~{}_0[P]\cap K\neq\emptyset\}=\{w_{\langle1,  r\rangle}~|~w\in K\}$$
is the set of $r$-blocks occurring in $K$,   and $Bl(K)=\bigcup_{r\in \N}Bl_r(K)$ is the set of all $K$-blocks.   For the number of blocks we write
$$\theta_r(K)=\card Bl_r(K),  ~\theta(K)=\theta_1(K).  $$
We shall often assume that we have some subshift $K$ without specifying $\Sigma_k^+$.   In such a case we use $\theta(K)$ as a bound for the number of states.

We write the topological entropy of $K$ as $h(K)=\htop(\sigma|K)$.   We have \cite[Proposition 16.  11]{DGS}:
\begin{equation}\label{entropy-K}
  h(K)=\lim_{r\to\infty}r^{-1}\log\theta_r(K).
\end{equation}

The symbol $\prec$ will express all kinds of occurrence of blocks:

a)  If $P,  Q\in Bl(M)$,   then $P\prec Q$ if $P$ is a subblock of $Q$,   or equivalently,   $\mu_Q(P)>0$.   If $P\prec Bl$,   $\omega\in \Sigma_k^+$,   then
$P\prec\omega$ if $P$ occurs in $\omega$ or equivalently,   $\omega\in\bigcup_{i=-\infty}^{+\infty}{}_i[P]$.   If $P\in Bl(M)$ and $K\subseteq \Sigma_k^+$ is subshift,   then $P\prec K$ if there is an $\omega\in K$ with $P\prec\omega$,   or equivalently,   $P\in Bl(K)$.

b) If $P\in Bl(M)$ and $\omega\in \Sigma_k^+$ (resp.   $K\subseteq \Sigma_k^+$ is a subshift),   then $P\underset{d}{\prec}\omega$ (resp.   $P\underset{d}{\prec}K$) if there is a $k\in\N$ such that the distances of the occurrences of $P$ in $\omega$ (resp.   in $\eta\in K$) are at most $k$,   i.  e.
$$O(\omega)\subseteq\bigcup_{i=0}^{k-1}{}_i[P]~(\text{resp.  }~K\subseteq\bigcup_{i=0}^{k-1}{}_i[P]).  $$
Because of the compactness of $K$,   $P\underset{d}{\prec} K\Leftrightarrow \forall~\eta\in K: P\underset{d}{\prec}\eta$.
We say that $P$ ``occurs densely".

c) If $P\in Bl(M)$,   $\omega\in \Sigma_k^+$ (resp.   $K\subseteq \Sigma_k^+$ is a subshift) and $\eps>0$,   then $P\underset{\eps-reg}{\prec}\omega$ (resp.   $P\underset{\eps-reg}{\prec}K$) if there is some $t>l(P)$ such that
$$Q_1,  Q_2\prec\omega~(\text{resp.  }~Q_1,  Q_2\prec K),  ~l(Q_i)=t\Rightarrow|\mu_{Q_1}(P)-\mu_{Q_2}(P)|<\eps.  $$
Equivalently,   $P\underset{\eps-reg}{\prec}\omega$ (resp.   $P\underset{\eps-reg}{\prec}K$) if there is some $s>l(P)$ such that
$$Q_1,  Q_2\prec\omega~(\text{resp.  }~Q_1,  Q_2\prec K),  ~l(Q_i)\geq s\Rightarrow|\mu_{Q_1}(P)-\mu_{Q_2}(P)|<\eps.  $$
We have the following Lemma \cite[Theorem 26.  7]{DGS}.
\begin{Lem}\label{Lemma-minimal}
The subshift $K\subseteq \Sigma_k^+$ is minimal if and only if $\forall~P\prec K: P \underset{d}\prec K$.

\end{Lem}
\subsection{Subshifts of Finite Type}

Let $B$ be a set of blocks occurring in $\Sigma_k^+$.   The subshift defined by excluding $B$ is the set
$$\Lambda_B:=\{x\in \Sigma_k^+~|~\text{no block}~\beta\in B~\text{occurs in}~x\}.  $$
Clearly $\Lambda_B$ is shift-invariant and closed for its complement is open.

A subshift $M$ is called a \emph{subshift of finite type} (f.  t.   subshift for short) if there exists a finite excluded block system for $M$.   We say that $N\in\N$ is an order of the f.  t.   subshift $M$ if there exists an excluded block system $B$ for $M$,   which only contains blocks of length $\leq N$.

If $P,  Q\prec M$,   a $M$-transition block from $P$ to $Q$ is a block $U\prec M$ such that $P\cdot U\cdot Q\prec M$.   We say $L\in\N$ is a transition length from $P$ to $Q$ if there exist $M$-transition blocks $U\in Bl_L(M)$ and $U'\in Bl_{L+1}(M)$ from $Q$ to $Q$.
Furthermore,   we say $L_0\in\N$ is a transition length for $M$ if for all $L\geq L_0$ and $P,  Q\prec M$,   $L$ is a transition length from $P$ to $Q$.

We have several equivalent descriptions.
\begin{Lem} \cite{DGS}\label{finite-type-subshift}
The following two conditions are equivalent.
\begin{description}
  \item[1)] $M$ is a f.  t.   subshift of order $N$.
  \item[2)] There exists a block system $B'$ of blocks of length $N$ such that
  $$\Lambda=\{(x_n)_{n\in\Z}~|~\forall n\in \Z: (x_n,  x_{n+1,  \cdots,  x_{n+N-1}})\in B'\}.  $$
  A set $B'$ of blocks defining a subshift $M$ in this way will be called a defining system of blocks for $M$.

\end{description}

\end{Lem}
\begin{Lem} \cite{DGS}\label{transitivity-for-symbolic}
The following conditions are equivalent:
\begin{enumerate}
  \item $(M,  \sigma)$ is topologically transitive.
  \item For every $i,  j$ occurring in $M$,   there exists an $n\in\N$ such that $\{x\in M~|~x_0=i,  x_n=j\}\neq\emptyset$.
\end{enumerate}
\end{Lem}

\begin{Lem} \cite{DGS}\label{mixing-for-symbolic}
The following three conditions are equivalent:
\begin{description}
  \item[a)] $M$ is an m.  f.  t.  -subshift;
  \item[b)] there exists a transition length;
  \item[c)] $M$ is topologically transitive and there are two blocks $P,  Q\prec M$ (both at least as long as the order of $M$) such that there exists a transition length from $P$ to $Q$.
\end{description}
\end{Lem}

\subsection{Basic Lemmas}
\begin{Lem} \cite[Lemma 26.  16]{DGS}\label{slight-extension}
Let $M$ be a m.  f.  t.  -subshift,   $K_1,  K_2\subseteq M$ disjoint subshifts with $K_1\neq\emptyset$ and $\eps>0$.
\begin{description}
  \item[a)] There is an m.  f.  t.  -subshift $M_1$ with
  $$K_1\subseteq M_1\subseteq M,  ~K_2\cap M_1=\emptyset~\textrm{and}~h(M_1)<h(K_1)+\eps.  $$
  \item[b)] If $C,  C'$ are finite sets of $M-$blocks and $K_1$ has one or both of the following properties
  $$P\in C\Rightarrow P\underset{d}{\prec} K_1;~P\in C'\Rightarrow P\underset{\eps-reg}{\prec} K_1,  $$
  then the corresponding properties can be achieved for $M_1$.
\end{description}

\end{Lem}
\begin{Lem}\label{pre-main-Lemma}
Let $M$ be a m.  f.  t.  -subshift,   $\nu\in \mathfrak{M}_{\sigma}(M)$ be ergodic and $0<\widehat{h}<h_{\nu}$.   Then for any $\widehat{t}\in \N$ and $\delta>0$,   there exists a subshift $\widehat{M}\subsetneq M$ and  $\widehat{r} \in\N$ such that
\begin{description}
    \item[i)] $|\mu_Q(P)-\nu(P)| < \delta$ for all $P\in \bigcup_{t\leq\widehat{t}}Bl_t(M)$ and $Q\in\bigcup_{r\ge\widehat{r}} Bl_r(\widehat{M})$ .
   \item[ii)] $h(\widehat{M})>\widehat{h}$.

\end{description}
\end{Lem}
\begin{proof}
Consider the following neighborhood $G$ of $\nu$:
$$G:=\{\mu\in\mathfrak{M}(M)~|~|\mu(P)-\nu(P)|<\delta~\text{for}~P\in \bigcup_{t\leq \widehat{t}}Bl_t(M)\}.  $$
Since $M$ has the basic-entropy-dense property,   we
obtain a subshift $\widetilde{M}\subseteq M$ with $h(\widetilde{M})>\widehat{h}$ and  $n_G \in\N$ such that
$\mathcal{E}_n(x)\in G$  for all $x\in \widetilde{M}$ and $n\geq n_G$.   This shows that $\widetilde{M}\neq M.  $
Then it is not hard to use \eqref{eq-P-Q-x} and Lemma \ref{slight-extension} to find the required $\widehat{M}\supseteq \widetilde{M}$ and $\widehat{r}$.

\end{proof}
Furthermore,   we have the following Lemma.
\begin{Lem}\label{main-Lemma}
Let $M$ be a m.  f.  t.  -subshift,   $\nu\in \mathfrak{M}_{\sigma}(M)$ be ergodic and $0<\overline{h}<h_{\nu}$.   Then for any $\overline{t}\in \N$ and $\delta>0$,   there exists a m.  f.  t-subshift $\overline{M}\subseteq X$ and  $\overline{r} \in\N$ such that
\begin{description}
  \item[1)] $P\in Bl_{\overline{t}}(M),  Q\in Bl_{\overline{r}}(\overline{M})\Rightarrow P\prec Q$.
  \item[2)] $|\mu_Q(P)-\nu(P)| < \delta$ for all $P\in \bigcup_{t\leq\overline{t}}Bl_t(M)$ and $Q\in\bigcup_{r\ge\overline{r}} Bl_r(\overline{M})$.
  \item[3)] $h(\overline{M})>\overline{h}$.

\end{description}
\end{Lem}
\begin{proof}
By Lemma \ref{pre-main-Lemma},   we obtain a subshift $\widehat{M}\subsetneq M$ and $\widehat{r}\in\N$ such that
\begin{description}
    \item[a)] $|\mu_Q(P)-\nu(P)| < \delta/2$ for all $P\in \bigcup_{t\leq\overline{t}}Bl_t(M)$ and $Q\in\bigcup_{r\ge\widehat{r}} Bl_r(\widehat{M})$ .
   \item[b)] $h(\widehat{M})>\frac{\overline{h}+h_\nu}{2}$.

\end{description}
Since $\widehat{M}\neq M$,   there exists an $s\in\N$ such that $Bl_s(\widehat{M})\neq Bl_s(M)$.   Now choose a block $A\in Bl_s(M)\setminus Bl_s(\widehat{M})$.   Since $M$ is topological mixing,   there is a transition length $L$ so large that for any $Q,  Q'\in Bl(M)$,   there is $U\in Bl_L(M)$ and $U'\in Bl_{L+1}(M)$ with $QUQ'\prec M$ and $QU'Q'\prec M$.   Furthermore,   $U,  U'$ can be chosen such that (by the topological mixing property)
\begin{equation}\label{P-prec-U}
  P\in Bl_{\overline{t}}(M)\Rightarrow P\prec U~\text{and}~ P\prec U'.
\end{equation}
Moreover,
\begin{equation}\label{A-U-U'}
  A\prec U~\text{and}~A\prec U'
\end{equation}
so that $U,  U'$ are `recognizable'.

For sufficient large $k\in\N$ such that
\begin{enumerate}
  \item
  \begin{equation}\label{eq-theta-m}
    \frac{\ln \theta_m(\widehat{M})}{m}\geq \htop(\widehat{M})~\text{when}~m\geq k.
  \end{equation}
    \item
  \begin{equation}\label{eq-k-L}
    \frac{L+1}{k+L+1}\leq \delta/2.
  \end{equation}
   \item
   \begin{equation}\label{eq-H-k-L}
     \frac{k}{k+L+1}\cdot\frac{\overline{h}+h_\nu}{2}>\overline{h}.
   \end{equation}
 \end{enumerate}
Define the subshift $\overline{M}$ having the following defining block systems:
\begin{eqnarray*}
  Bl_{k+L+1}(\overline{M}) &=& \{P\in Bl_{k+L+1}(M):P\prec QUR,   ~\text{where}~Q,  R~\text{run through}~Bl_k(\widehat{M})~\text{and} \\
   & &~U\in Bl_L(M)\cup Bl_{L+1}(M)~\text{is a}~M~\text{transition block from}~Q~\text{to}~R~\text{satisfying}~\eqref{P-prec-U}~\text{and}~\eqref{A-U-U'}\}.
\end{eqnarray*}
By Proposition \ref{finite-type-subshift},   \ref{transitivity-for-symbolic} and \ref{mixing-for-symbolic},   $\overline{M}$ is a m.  f.  t.   subshift.   By \eqref{P-prec-U},   one obtains Property 1) of $\overline{M}$.   Let $\overline{r}>\widehat{r}$ be sufficient large.   Moreover,   one observes the fact every $k+L+1$-block in $\overline{M}$ carries a $k$-block in $\widehat{M}$.   Then one uses the property a) of $\widehat{M}$ and \eqref{eq-k-L} to check property 2).   Finally,   let us calculate the topological entropy of $\overline{M}$ to substantiate property 3) of $\overline{M}$.    Indeed,   for any $l\in\N$,   there are blocks of the following form
$$Q_1U_1R_1U_1'Q_2U_2R_2U_2'\cdots Q_lU_lR_l$$
occurring in $\overline{M}$,   where $Q_i,  R_i\in Bl_{k}(\widehat{M})$ and $U_i,  U_i'\in Bl_L(\widehat{M})\cup Bl_{L+1}(\widehat{M})$ are transition blocks.   So one infers by \eqref{eq-theta-m} and \eqref{eq-H-k-L} that
$$\htop(\overline{M})=\lim_{r\to\infty}\frac{\ln \theta_r(\overline{M})}{r}=\lim_{l\to\infty}\frac{\ln \theta_{2l(k+L+1)}(\overline{M})}{2l(k+L+1)}\geq \lim_{l\to\infty}\frac{2l\theta_{k}(\widehat{M})}{2l(k+L+1)}\geq \frac{k}{k+L+1}\cdot\frac{\overline{h}+h_\nu}{2}>\overline{h}.  $$

\end{proof}

\subsection{Proof of Proposition \ref{prop-symbolic}}

For $k=1,$ one can follow the idea of \cite{Grillenberger,DGS} to give the proof.
For completeness of this paper, here we give the proof details  for general $k\geq 2$ by  modifying the proof of \cite{Grillenberger,DGS}.
Let \begin{equation}\label{selection-of-eps}
   0<\eps<\min\{\rho(\mu^i,  \mu^j),  \dist(\mu^a,  \cov\{\mu^b\}_{1\leq b\leq k,  b\neq a}):1\leq i<j\leq k,  1\leq a\leq k\},
 \end{equation}

For each $1\leq j\leq k$,   choose a strictly decreasing sequence $\{h_i^j\}_{i=0}^{\infty}$ such that
$$h_0^j=h_{\mu^j}~\textrm{and}~\lim_{i\to\infty}h_i^j=h_{\mu^j}-\eta.  $$
We now inductively construct $k+1$ sequences of m.  f.  t.   subshifts $\{M_i\}_{i=0}^\infty$ and $\{M_i^j\}_{i=0}^\infty$,   $j=1,  \cdots,  k$.   Let $M_0=M$.   According to Lemma \ref{measure-on-symbolic},   for each $i\in\N$,   there exist  $\overline{t}_{i-1}>0$ and $\overline{\delta}_{i-1}>0$ such that
\begin{equation}\label{t-delta-i-1}
  \lambda,  \tau\in \mathfrak{M}(M),  ~|\lambda(P)-\tau(P)|<\overline{\delta}_{i-1}~\forall~P\in Bl_{\overline{t}_{i-1}}(M_{0})\Rightarrow \rho(\lambda,  \tau)<3^{-i-1}\eps.
\end{equation}
Let $t_0=\max\{\overline{t}_0,  1\}$ and $\delta_0=\min\{\overline{\delta}_0,  1\}$.   Then using
Lemma \ref{main-Lemma},   we obtain $k$ m.  f.  t.   subshift $M_0^1,  M_0^2,  \cdots,  M_0^k$ and $\overline{r}_0\in\N$ satisfying that
\begin{description}
    \item[i)] $|\lambda_0^j(P)-\mu^j(P)|<\delta_0$ for all $\lambda_0^j\in \mathfrak{M}_{\sigma}(M_0^j)$ and $P\in \bigcup_{t\leq t_0}Bl_t(M_0)$ ($j=1,  \cdots,  k$).   Consequently,
        \begin{equation}\label{eq-nu-0-j}
          \rho(\lambda_0^j,  \mu^j)<3^{-1}\eps~\textrm{for all}~\lambda_0^j\in \mathfrak{M}_{\sigma}(M_0^j).
        \end{equation}

  \item[ii)] $h(M_0^j)>h_1^j$ ($j=1,  \cdots,  k$).

\end{description}
Therefore,   \eqref{eq-nu-0-j} and the selection of $\eps$ in \eqref{selection-of-eps} indicate that
\begin{equation}\label{rho-nu-0-1-2}
  \rho(\lambda_0^s,  \lambda_0^t)>3^{-1}\eps>0~\text{for any}~\lambda_0^s\in \mathfrak{M}_{\sigma}(M_0^s)~\textrm{and}~\lambda_0^t\in \mathfrak{M}_{\sigma}(M_0^t),  ~1\leq s< t\leq k.
\end{equation}
Consequently,   $M_0^s\cap M_0^t=\emptyset$.   For otherwise the nonempty intersection is a closed invariant set which supports some invariant measure belonging to both $\mathfrak{M}_{\sigma}(M_0^s)$ and $\mathfrak{M}_{\sigma}(M_0^t)$,   a contradiction to \eqref{rho-nu-0-1-2}.

Suppose in the $(i-1)$th step of the construction we have obtained  a m.  f.  t.  -subshift $M_{i-1}$ and disjoint  m.  f.  t.  -subshifts $\{M_{i-1}^j\}_{j=1}^k$ of $M_{i-1}$ with
\begin{equation}\label{close-to-full-entropy}
    h(M_{i-1}^j)>h_{i-1}^j,  ~j=1,  \cdots,  k.
  \end{equation}
Then there exist $k$ ergodic measures $\nu_{i-1}^j\in \mathfrak{M}_{\sigma}(M_{i-1}^j)(j=1,  \cdots,  k)$ such that
$$h_{\nu_{i-1}^j}>h_{i-1}^j,  ~j=1,  \cdots,  k.  $$
We now choose $t_{i-1}>\max\{\overline{t}_{i-1},  2^{i-1}\}$ so large that
  $$Bl_{t_{i-1}}(M_{i-1}^s)\cap Bl_{t_{i-1}}(M_{i-1}^t)=\emptyset,  1\leq s<t\leq k$$
  and
  \begin{equation}\label{entropy-t-i}
    \log \theta_{t_{i-1}}(M_{i-1}^j)>t_{i-1}h_{i-1}^j,  ~j=1,  \cdots,  k.
  \end{equation}
Here \eqref{entropy-t-i} comes from \eqref{entropy-K}.
For $h_i^j$ and $0<\delta_{i-1}<\min\{\overline{\delta}_{i-1},  2^{-i+1}\}$,   a convenient use of Lemma \ref{main-Lemma} produces $k$ m.  f.  t.   subshifts $\{M_i^j\}_{j=1}^k$ and $\overline{r_i}>t_{i-1}$ such that
\begin{description}
  \item[i)]
  \begin{equation}\label{P-Q}
    P\in Bl_{t_{i-1}}(M_{i-1}),  Q\in Bl_{\overline{r_i}}(M_i^j)\Rightarrow P\prec Q(j=1,  \cdots,  k).
  \end{equation}
  \item[ii)] $|\mu_{Q}(P)-\nu_{i-1}^j(P)|<\delta_{i-1}$ for all $P\in \bigcup_{t\leq t_{i-1}}Bl_t(M_{i-1}),  Q\in\bigcup_{r\geq\overline{r_i}}Bl_r(M_i^j)$.   Consequently,
   \begin{equation}\label{nu-i-nu-i-1}
     |\lambda_i^j(P)-\nu_{i-1}^j(P)|<\delta_{i-1}~\textrm{for all}~\lambda_i^j\in \mathfrak{M}_{\sigma}(M_i^j)~\textrm{and}~P\in \bigcup_{t\leq t_{i-1}}Bl_t(M_{i-1})(j=1,  \cdots,  k).
   \end{equation}
  \item[iii)] $h(M_i^j)>h_i^j$ ($j=1,  \cdots,  k$).

\end{description}
Moreover,   note that   $M_i^1\cup \cdots \cup M_i^k\neq M_{i-1}$.   So $\overline{r}_i$ can be chosen such that
\begin{equation}\label{1-2-M}
  Bl_{\overline{r_i}}(M_i^1\cup M_i^2\cup\cdots M_i^k)\neq Bl_{\overline{r_i}}(M_{i-1}).
\end{equation}

Now choose a block $C_i\in Bl_{\overline{r_i}}(M_{i-1})\backslash Bl_{\overline{r_i}}(M_i^1\cup\cdots\cup M_i^k)$ and a $M_{i-1}$ transition length $L_i$ which is so large that for $P^s\in Bl(M_i^s)$ and $P^t\in Bl(M_i^t)$,   $1\leq s\neq t\leq k$,   there are blocks $U_i^{st}\in Bl_{L_i}(M_{i-1})$ with
\begin{equation}\label{eq-a}
   P^sU_i^{st}P^t\prec M_{i-1}.
    \end{equation}
    and
   \begin{equation}\label{eq-b}
      C_i\prec U_i^{st}~\text{but only once,   namely at the place}~[\frac{L_i}{2}].
    \end{equation}
Choose an integer $r_i>4(\overline{r_i}+L_i)\cdot\eps_{i-1}^{-1}$ and bigger than the order of $M_{i-1}$.
Then choose $k$ fixed blocks $S_i^j\in Bl_{r_i}(M_i^j)(j=1,  \cdots,  k)$.   Furthermore,   choose $(n^2-n)$ fixed transition blocks $U_i^{st}(1\leq s\neq t\leq k)$ satisfying \eqref{eq-a},   \eqref{eq-b} and such that
$$S_i^sU_i^{st}S_i^t\prec M_{i-1},  1\leq s\neq t\leq k.  $$

Let $M_i$ is the m.  f.  t.  -subshift having
  $$Bl_{r_i}(M_i)=Bl_{r_i}(M_i^1\cup\cdots\cup M_i^k)\bigcup\{Q\in Bl_{r_i}(M_{i-1})~|~Q\prec S_i^sU_i^{st}S_i^t,  1\leq s\neq t\leq t\}$$
  as defining block system.   By Proposition \ref{finite-type-subshift},   \ref{transitivity-for-symbolic} and \ref{mixing-for-symbolic},   $M_i$ is a m.  f.  t.   subshift.

We now define $\overline{M}$ as
$$\overline{M}:=\bigcap_{i\geq 1}M_i.  $$
For each $1\leq j\leq k$,   by \eqref{nu-i-nu-i-1},   \eqref{t-delta-i-1} and the selection of $\delta_i$,   we see
\begin{equation}\label{rho-nu-i-i-1}
  \rho(\nu_i^j,  \nu_{i-1}^j)<3^{-i-1}\eps,  i\in\N.
\end{equation}
So $\{\nu_i^j\}_{i\in\N}$ constitutes a Cauchy sequence.   Let $\nu^j$ denotes its limit.
\begin{Lem}
We have the following facts
\begin{description}\label{lem-four-facts}
  \item[(1)] $\overline{M}$ is minimal.
  \item[(2)] Each $\nu^j$ is $\sigma$-invariant and supported on $\overline{M}$.
  \item[(3)] $h_{\nu^j}>h_{\mu^j}-\eta,  j=1,  \cdots,  k$.
  \item[(4)]
  \begin{equation}\label{eq-rho-nu-j-mu-j}
    \rho(\nu^j,  \mu^j)<\eps/2,  j=1,  \cdots,  k.
  \end{equation}
 \end{description}

\end{Lem}
\begin{proof}
\begin{description}
  \item[(1)] Observe that $\overline{M}$ can be characterized as $Bl_{t_i}(\overline{M})=Bl_{t_i}(M_i)$.   Indeed,   $Bl_{t_i}(\overline{M})\subseteq Bl_{t_i}(M_i)$ is obvious since $\overline{M}\subseteq M_i$.   On the other hand,   for every $P\in Bl_{t_i}(M_i)$,   one has $P\prec M_{i+1}$.   Inductively,   one gets that $P\prec M_j$ for any $j\ge i$.   Consequently,   one has $P\prec \overline{M}=\bigcap_{j\geq i}M_j$ since $\{M_i\}_{i\in \N}$ is a decreasing sequence.   Therefore,   one has $Bl_{t_i}(\overline{M})\supset Bl_{t_i}(M_i)$.   Then minimality of $\overline{M}$ follows from \eqref{P-Q} and Lemma \ref{Lemma-minimal}.
  \item[(2)] Note that the space of $\sigma$-invariant measures on $M$ are closed,   so each $\nu^j$ is $\sigma$-invariant.   Moreover,   for each $1\leq j\leq k$ and $i\in\N$,   $1=\limsup_{i\to\infty}\nu_i^j(M_i)\leq \nu^j(M_i)$ \cite{Walters}.   So $\mu^j(M_i)=1$ for each $i$,   which implies that $\nu^j(\overline{M})=\nu^j(\bigcap_{i\in\N}M_i)=1$.   However,   $\overline{M}$ is minimal.   Thus $S_{\nu^j}=\overline{M}$.
  \item[(3)] Since $M$ is expansive,   the entropy map is upper semi-continuous \cite{Walters}.   So one has
  $$h_{\nu^j}\geq\limsup_{i\to\infty}h_{\nu_i^j}>h_{\mu^j}-\eta.  $$
  \item[(4)] Indeed,   \eqref{eq-nu-0-j} and \eqref{rho-nu-i-i-1} indicate that
   $\rho(\nu^j,  \mu^j)\leq \sum_{i=1}^{\infty}3^{-i}\eps=\eps/2.  $
\end{description}
\end{proof}
Let $K:=\cov\{\nu^j\}_{j=1}^k$.
\begin{Lem}\label{lem-M-sigma-subseteq-K}
$\M_{\sigma}(\overline{M})\subseteq K$.
\end{Lem}
\begin{proof}
Since $\M_{\sigma}(\overline{M})$ and $K$ are convex,   we only need to show that for any ergodic measure $\nu\in\mathfrak{M}_{\sigma}(\overline{M})$,   $\nu\in K$.   Let $\kappa>0$.   There exist $t\in\N$ and $\eta>0$ such that for $\alpha,  \alpha'\in\mathfrak{M}_{\sigma}(\overline{M})$:
\begin{equation}\label{alpha-rho}
  \forall~P\in Bl_t(M):|\alpha(P)-\alpha'(P)|<4\eta\Rightarrow \rho(\alpha,  \alpha')<\kappa.
\end{equation}
Let $i_0$ is so large that $8\cdot2^{-i_0+1}<\eta$ and $2^{i_0-1}\geq t$.   Then for all $i\geq i_0,  r\geq \overline{r}_i$,   one has
$$t\leq 2^{i_0-1}\leq 2^{i-1}<t_{i-1}~\text{and}~8\eps_{i-1}<8\cdot2^{-i+1}\leq 8\cdot2^{-i_0+1}<\eta.  $$
Therefore,   for any
$P\in Bl_t(M),  Q\in Bl_r(M_i^j)$,   one has
$$|\mu_Q(P)-\nu_{i-1}^j(P)|<\delta_{i-1}<\eta~\text{and thus}~|\nu_i^j(P)-\nu_{i-1}^j(P)|<\delta_{i-1}<\min\{2^{-i+1},  \eta\}.  $$
Consequently,
\begin{eqnarray*}
   & & |\mu_Q(P)-\nu^j(P)| \\
   &\leq& |\mu_Q(P)-\nu_{i-1}^j(P)|+\sum_{k=i}^{\infty}|\nu_{k-1}^j(P)-\nu_k^j(P)| \\
   &\leq& \eta+\sum_{k=i}^{\infty}2^{-k+1}=\eta+2\cdot2^{-i+1}<\eta+\eta=2\eta.
\end{eqnarray*}

Now choose a generic point $w\in\overline{M}$ of $\nu$.
Then there is an $n_0\in\N$ such that for all $n\geq n_0$ and all $P\in Bl_t(M)$
\begin{equation}\label{w-generic}
  |\mu_{w_{\langle0,  n-1\rangle}}(P)-\nu(P)|=\left|\frac1n\sum_{i=0}^{n-1}1_P(\sigma^i(\omega))-\nu(P)\right|<\eta.
\end{equation}
Choose an $i\geq i_0$ such that $\overline{r}_i\geq n_0$.   Then choose an $n\in \N$ such that
$$\overline{r}_i<n<\eps_{i-1}\cdot r_i.  $$
This is possible by the selection of $r_i$.
Since all the transition blocks $U_i^{st}$ are recognizable by the block $C_i$,   the sequence $w$ can be split up uniquely into the transition blocks and into blocks belonging to the subshifts $M_i^j$,   the latter with length at least $\frac34r_i$,   so we have the following two cases:
\begin{description}
  \item[a)] $w_{\langle0,  n\rangle}$ falls completely into a $M_i^j$-piece.   Then $w_{\langle0,  n-1\rangle}\in Bl_{n}(M_i^j)$ with $n> \overline{r}_i$.   So
      $$|\mu_{w_{\langle0,  n-1\rangle}}(P)-\nu^j(P)|<2\eta.  $$
      This along with \eqref{w-generic} indicates that
      $$|\nu(P)-\nu^j(P)|<3\eta~\text{for any}~P\in Bl_t(M).  $$
      Hence  $d(\nu,  \nu^j)<\kappa$ by \eqref{alpha-rho}.
  \item[b)] $w_{\langle0,  n-1\rangle}$ overlaps one of the transition blocks $U$.   Then at the right of $U$ begins a $M_i^j$-piece.   For the relative frequency $\mu_{w_{\langle0,  n+L_i+[\frac{r_i}{2}]-1\rangle}}$,   only the part $w_{\langle n+L_i,  n+L_i+[\frac{r_i}{2}]-1\rangle}$ which belongs to a $M_i^j$-block is essential.   This is due to $(n+L_i)<\eps_{i-1}\cdot r_i+\eps_{i-1}\cdot r_i=2\eps_{i-1}\cdot r_i$.   More exactly,   for $P\in Bl_t(\overline{M})$,
      $$\mu_{w_{\langle0,  n+L_i+[\frac{r_i}{2}]-1\rangle}}(P)=\frac{(n+L_i)\mu_{w_{\langle0,  n+L_i+r-1\rangle}}(P)+([\frac{r_i}{2}]-r+1)\mu_{w_{\langle n+L_i,  n+L_i+[\frac{r_i}{2}]-1\rangle}}(P)}{n+L_i+[\frac{r_i}{2}]-r+1}.  $$
      This yields that
      \begin{equation}\label{estimation-1}
        \left|\mu_{w_{\langle0,  n+L_i+[\frac{r_i}{2}]-1\rangle}}(P)-\mu_{w_{\langle n+L_i,  n+L_i+[\frac{r_i}{2}]-1\rangle}}(P)\right|<2\cdot\frac{n+L_i}{[\frac{r_i}{2}]}<8\eps_{i-1}<\eta.
      \end{equation}
      Now we employ the argument in case $a)$ and get that
      \begin{equation}\label{estimation-2}
        \left|\mu_{w_{\langle n+L_i,  n+L_i+[\frac{r_i}{2}]\rangle}}(P)-\nu^j(P)\right|<2\eta.
      \end{equation}
      In the light of \eqref{w-generic},   \eqref{estimation-1} and \eqref{estimation-2},   one gets $|\nu(P)-\nu^j(P)|<4\eta$ which implies again $d(\nu,  \nu^j)<\kappa$ by \eqref{alpha-rho}.
\end{description}
Since $\kappa>0$ can be chosen arbitrarily small,   one sees that $\nu$ must coincide with some $\nu^j$ with $1\leq j\leq k$.   This proves that there are no other ergodic measures supported on $\overline{M}$ than $\{\nu^j\}_{j=1}^k$.   Since $K$ is convex,   we obtain $\M_{\sigma}(\overline{M})\subseteq K$.
\end{proof}
\begin{Lem}
$\M_{\sigma}(\overline{M})= K$.   In particular,   $\overline{M}$ supports exactly $k$ ergodic measures $\{\nu^j\}_{j=1}^k.  $
\end{Lem}
\begin{proof}
By \eqref{eq-rho-nu-j-mu-j},   for any $1\leq a\leq k$,
\begin{equation}\label{cov-subset}
  \dist(\cov\{\mu^b\}_{1\leq b\leq k,  b\neq a},  \cov\{\nu^b\}_{1\leq b\leq k,  b\neq a})<\eps/2.
\end{equation}
This along with \eqref{selection-of-eps} implies for any $1\leq a\leq k$ that
\begin{eqnarray}
   & & \dist(\nu^a,  \cov\{\nu^b\}_{1\leq b\leq k,  b\neq a})\nonumber \\
   &\geq& \dist(\mu^a,  \cov\{\mu^b\}_{1\leq b\leq k,  b\neq a})-\dist(\nu^a,  \mu^a)-\dist(\cov\{\mu^b\}_{1\leq b\leq k,  b\neq a},  \cov\{\nu^b\}_{1\leq b\leq k,  b\neq a}) \nonumber\\
   &>&  \eps-\eps/2-\eps/2=0.  \label{nu-a-neq}
\end{eqnarray}
So if $\M_{\sigma}(\overline{M})\neq K$,   then by Lemma \ref{lem-M-sigma-subseteq-K},   there is some $1\leq a\leq k$ such that
$$\nu^a\in\cov\{\nu^b\}_{1\leq b\leq k,  b\neq a},  $$
contradicting \eqref{nu-a-neq}.
\end{proof}
Therefore,   $\overline{M}$ has exactly $k$ ergodic measures $\{\nu^j\}_{j=1}^k$ and the proof of Proposition \ref{prop-symbolic} is completed.  \qed

\bigskip

\bigskip

{\bf Acknowledgements.  } The authors are grateful to Prof.   Yu Huang,   Wenxiang Sun and Xiaoyi Wang for their numerous remarks and fruitful discussions.   The research of X.   Tian was  supported by National Natural Science Foundation of China (grant no.   11671093).

\end{document}